\documentclass[a4paper,reqno]{amsart}
\newtheorem{theorem}{Theorem}[section]

\newtheorem{lemma}[theorem]{Lemma}

\newtheorem{corollary}[theorem]{Corollary}

\def\beq{\begin{equation}}
\def\eeq{\end{equation}}
\def\be{\begin{equation*}}
\def\ee{\end{equation*}}
\usepackage{amssymb}
\usepackage{amsfonts}

\title{On limits of betweenness relations}
\author[D. Bradley-Williams and J. K. Truss]{David Bradley-Williams and John K. Truss}
\begin{document}

\maketitle

\setcounter{footnote}{1}\footnotetext{
This paper is based on part of the first author's PhD thesis at the University of Leeds, 2014, which received DTG funding from the EPSRC. He would
particularly like to thank Dugald Macpherson for introducing him to the topic of this paper, and for
many helpful discussions during his time as a PhD student. He is partially supported
by the research training group GRK 2240: \emph{Algebro-Geometric Methods in Algebra, Arithmetic and Topology,}
funded by the DFG. The authors are grateful to the anonymous referee for suggestions that improved the presentation of this paper.\\
keywords: Jordan groups, betweenness relation, limit of betweenness relations, tree of $B$-sets, semilinear order, infinite primitive permutation groups.}
\newcounter{number}

\begin{abstract} We give a flexible method for constructing a wide variety of limits of betweenness relations. This 
unifies work of Adeleke, who constructed a Jordan group preserving a limit of betweenness relations, and Bhattacharjee 
and Macpherson who gave an alternative method using a Fra{\"i}ss\'e-type construction. A key ingredient in their 
work is the notion of a tree of B-sets. We employ this, and extend its use to a wider class of examples.
\end{abstract}

\section {Introduction}

A permutation group $G$ acting on a set $\Omega$ is said to be {\em $k$-transitive} (or to act {\em $k$-transitively}) 
if for any two $k$-tuples of distinct elements of $\Omega$ there is an element of $G$ taking the first to the second, 
and it is {\em highly transitive} if it is $k$-transitive for all positive integers $k$. If for any two $k$-element 
sets there is a group element taking the first to the second, then $G$ is said to be {\em $k$-homogeneous}, and it is
{\em highly homogeneous} if it is $k$-homogeneous for all $k$. If $G$ acts transitively on 
$\Omega$, then $G$ is said to be a {\em Jordan group} if there is a proper subset $\Gamma$ of $\Omega$ with 
$|\Gamma| > 1$ satisfying the following: the pointwise stabilizer $G_{(\Omega \setminus \Gamma)}$ in $G$ of $\Omega \setminus \Gamma$ acts transitively on $\Gamma$, and for all $k \ge 0$ such that $G$ acts $(k+1)$-transitively on $\Omega$, we have $|\Omega \setminus \Gamma| > k$. Such $\Gamma$ is called a {\em Jordan set} for $G$. 
(The condition that $|\Omega \setminus \Gamma| > k$ is required for non-triviality, since otherwise, by 
$(k+1)$-transitivity, any $\Gamma$ satisfying $|\Omega \setminus \Gamma| \le k$ would have to count as a Jordan set.) 
We recall that a permutation group is said to be {\em primitive} if it preserves no non-trivial equivalence relation. 
If $G$ is a primitive Jordan group acting on $\Omega$, and $\Gamma$ is a Jordan set such that 
$G_{(\Omega \setminus \Gamma)}$ acts primitively on $\Gamma$, then we say that $\Gamma$ is a {\em primitive} Jordan set 
for $G$.

This paper makes a contribution to the study of infinite primitive Jordan groups, which was initiated by S. A. Adeleke and 
P. M. Neumann \cite{Adeleke3}, and continued by S. A. Adeleke and H. D. Macpherson \cite{Adeleke2}. In the latter paper, a list of possible 
structures which can be preserved by a primitive Jordan group was given, based on various hypotheses on transitivity, 
homogeneity and primitivity. According to \cite{Adeleke2} Theorem 1.0.1, an infinite Jordan group which is not highly 
transitive preserves one of the following structures on $\Omega$:

1. a dense linear order,

2. a dense circular order,

3. a dense linear betweenness relation,

4. a dense separation relation,

5. a dense semilinear order,

6. a dense general betweenness relation,

7. a $C$-relation,

8. a $D$-relation,

9. a Steiner system,

10. a limit of betweenness relations,

11. a limit of $D$-relations,

12. a limit of Steiner systems.

We do not define all of these, but just the ones we need in this paper. In fact we focus on limits of betweenness relations, and to 
explain this, we shall also need to understand betweenness relations and semilinear orders (trees). Broader details regarding the theory, history, and applications concerning Jordan groups can be found in the introduction of \cite{Almazaydeh}.

In \cite{Adeleke2} the precise nature of the structures arising in the last three cases of this list was left rather incomplete, though 
this was elucidated in further work. In \cite{Adeleke1} examples were constructed of Jordan groups preserving limits of $B$-relations and 
limits of $D$-relations. Adeleke also constructed an example of a group preserving a limit of Steiner systems %(see \cite{Adeleke2}page 65)
\cite{Adeleke0}, which was followed by further constructions of towers of Steiner systems by K. Johnson \cite{Johnson}. In \cite{Bhattacharjee} an alternative method for constructing Jordan groups preserving a limit of $B$-relations was given, 
based on a `generic' construction. The authors use the important notion of a finite `tree of $B$-sets' to act as an approximation 
to a desired structure. They also used a new kind of ternary relation, termed an `$L$-set' or `$L$-relation', which is a first order 
structure somehow encoding a tree of $B$-sets. The key steps were verification of a suitable amalgamation property for trees of 
$B$-sets (equivalently for $L$-sets), and an appeal to a Fra\"iss\'e-type method to obtain a limiting structure. This was originally 
given as an $L$-structure, but the (now infinite) tree of $B$-sets was recoverable, and gave rise to the desired Jordan group. In related recent work \cite{Almazaydeh}, A. I. Almazaydeh and H. D. Macpherson have constructed a primitive Jordan group preserving a limit of $D$-relations as the automorphism group of a relational structure. For this the authors develop an analogous notion of `tree of $D$-sets'.

In \cite{Bradley}, an attempt was made to unify the approaches of Adeleke and Bhattacharjee--Macpherson by a suitable adaptation of the 
tree of $B$-sets technique. The idea was also that the method should apply much more widely, with the final (infinite) tree of $B$-sets
being constrained to have particular forms, not just of the generic kind in \cite{Bhattacharjee}. The main case presented there was
where the tree is a so-called ${\mathbb N}^+$-tree, being one in which the maximal chains are all isomorphic to the set of natural
numbers, going downwards. In order to capture this situation, the trees of $B$-sets were endowed with unary predicates, called
`depth' predicates, measuring how far from the top each node was. Many of the proposed steps were similar to \cite{Bhattacharjee}, 
though the verification of amalgamation differed somewhat.

Here we complete the picture as begun in \cite{Bradley}, and apply similar methods to a much wider class of trees. In fact we can
handle $2^{\aleph_0}$ pairwise non-isomorphic examples. The role of the depth predicates is now assumed by `colours' assigned to nodes, 
both in the finite trees of $B$-sets forming the approximations, and the final infinite limit of the construction. A key difference 
from the example given in \cite{Bhattacharjee} is that none of our limiting structures is $\aleph_0$-categorical, whereas that given in \cite{Bhattacharjee} {\em is}. 

One major difference in our presentation is that we approach our Fra\"iss\'e-style limit more directly, via the trees of $B$-sets, and 
rather than appealing to a version of Fra\"iss\'e's Theorem for first order structures as given in \cite{Evans}, we perform its inductive construction explicitly. 
This means that we can at least initially avoid the need for `reconstructing' the tree of $B$-sets from the $L$-relation, which is an 
important concern in \cite{Bhattacharjee} (reconstructing is however needed in the final section, to show that we have found
$2^{\aleph_0}$ non-isomorphic examples). It is however important that the $L$-structure is still present, and that is because by 
definition, a Jordan group is a {\em permutation} group, so must act on some set, and this will be the domain of the $L$-relation. In 
giving this construction, we just give the action in one direction; that is, we show how the automorphism group of the (limiting) 
tree of $B$-sets acts on the $L$-structure, but we do not show that all automorphisms of the $L$-structure arise in this way.

The paper is organized in the following sections.

In section 2, preliminaries are given concerning semilinear orders, $B$- and $C$-relations, and combinatorial trees, and the (rather
complicated) definition of `limit of betweenness relations'.

In the next section, we introduce the main tool, namely trees of $B$-sets. In our case, unlike in \cite{Bhattacharjee}, we fix our 
`ambient' tree in advance, which comes along with a colouring by unary predicates, first explaining precisely which trees we are
allowing. We give a number of basic results about finite trees of $B$-sets, and define their associated $L$-relations. These
definitions then extend to the infinite case.

In \cite{Bhattacharjee} and \cite{Bradley}, additional relational symbols were employed in order to help control amalgamation. We
replace the need for these by considering a notion of `strong substructure'. In the simplest case, for finite $B$-relations of
positive type, this means that the centroid of any three vertices lying in the smaller structure is the same whichever structure
is used to calculate it. There are corresponding notions of `strong substructure' and `strong embedding' for trees of $B$-sets.

In section 4, we concentrate on developing the amalgamation machinery, which is done entirely for strong embeddings. As in the 
earlier treatments, the main cases to be covered are 1-point extensions, and then this can be lifted to the general finite case.
Theorem \ref{4.6} is the main result of this section, which establishes the existence of a suitable limit tree of $B$-sets over the
given ambient tree, in the style of Fra\"iss\'e.

The next section presents a detailed analysis of the structure just constructed. Techniques are required to enable us to handle the
limit structure, and the section moves towards the verification that we have indeed constructed a primitive Jordan group, which 
preserves a limit of betweenness relations. To make the enterprise worthwhile, it is of course important that the group really 
{\em is} `irregular', meaning that it doesn't preserve any of the other familiar relations.

Since one our main motivations was to extend the results and methods of both \cite{Bhattacharjee} and \cite{Adeleke1}, we devote
the first part of section 6 to showing how our structure is exactly the same as Adeleke's in the case of ${\mathbb N}^+$-trees 
(as in \cite{Bradley}). Finally we show that we really do have $2^{\aleph_0}$ pairwise non-isomorphic examples. For this we 
have to recover the original tree from the permutation group. The method for this is as described in \cite{Bhattacharjee}, except
that we have to start from knowing {\em only} the permutation group, and so the first step is to show that the $L$-relation can 
be recognized in the permutation group, before invoking the machinery for reconstructing the tree.

\section {Preliminaries}

We begin by defining the two main kinds of relation that we shall need.

A {\em semilinear order} is a partially ordered set $(T, \le)$ such that for each $x \in T$, $\le$ linearly orders 
$\{y \in T: y \le x\}$, and such that $\forall x \forall y \exists z(z \le x \wedge z \le y)$ (without the second condition, it would 
be a `forest'). This is sometimes also called a `lower' semilinear order, since it is linear going downwards, or a `tree' (though this 
has many meanings, so we shall try to be unambiguous in our usage). In a general semilinear order, some or all meets may exist. 
Precisely, if for $x, y \in T$, there is a greatest lower bound, then this is unique and is called the {\em meet} of $x$ and $y$, 
and may be denoted by $x \curlywedge y$; we reserve $\wedge$ for conjuction in formulas. We have given the definition in terms of the reflexive ordering, but we may also use the 
corresponding strict relation given by $x < y$ if $x \le y \wedge x \neq y$. There is a least semilinear order containing a given 
one in which all meets exist, obtained by adjoining new points corresponding to pairs in $T$ which previously had no meet, and this 
is written $T^+$. A point of $T^+$ is said to be a {\em ramification point} if it is the meet of two incomparable members of $T$.

If $t \in T^+$, then we may define {\em cones} at $t$ to be equivalence classes of points strictly above $t$ under the relation 
$x \sim y$ if for some $z > t$, $z \le x \wedge z \le y$ (one verifies easily that the fact that this is an equivalence relation follows 
from semilinearity---in a general partial order it may not be one). The {\em ramification order} of $t$ is defined to be the number of 
cones at $t$. Note that this is greater than 1 if and only if $t$ is a ramification point, and it is 0 if and only if $t$ is a leaf 
(i.e. maximal). In \cite{Droste}, a class of sufficiently transitive semilinear orders is constructed and classified. These are countable, 
and 3-set-transitive, meaning that for any two isomorphic 3-element substructures there is an automorphism taking the first to the second 
(though not necessarily extending the given map). This entails that the ramification points either all lie in $T$, or none do (i.e. they 
all lie in $T^+ \setminus T$) and that all ramification orders of ramification points are equal. There are countably many such 
structures up to isomorphism, all having maximal chains isomorphic to $\mathbb Q$, determined by whether or not the ramification points 
are in $T$ (`positive type') or not (`negative type'), and what the ramification order is. In this paper we shall just be concerned with 
the case of positive type with ramification order $\aleph_0$, which is the most typical, or `generic' case, and which we denote by $K$
throughout.

A {\em $B$-relation} is a ternary relation on a non-empty set $X$, which we write with a semicolon after the first variable, thought 
of as saying that the first point lies between the other two, satisfying the following three properties:

(B1) $B(x ;y, z) \to B(x; z, y)$,

(B2) $B(x; y, z) \wedge B(y; x, z) \leftrightarrow x = y$,

(B3) $B(x; y, z) \to B(x; y, w) \vee B(x; w, z)$.

We also call $(X, B)$ a {\em $B$-set}. It is {\em nontrivial} if $|X| \ge 2$. It is called a {\em betweenness relation} if in addition

(B4) $\neg B(x; y, z) \to (\exists w \neq x)(B(w; x, y) \wedge B(w; x, z))$.

(B5) $B$ is said to be of {\em positive type} if, for any $x$, $y$, $z$ for which the $B$-relation does not hold in any order, there 
is $u$ such that $B(u; x, y) \wedge B(u; x, z) \wedge B(u; y, z)$.

We say that a set $\{x, y, z\}$ is {\em incomparable} if the $B$-relation does not hold for them in any order, and an element $u$ as in
the conclusion of (B5) is called a {\em centroid} of $x$, $y$, and $z$. Thus (B5) says that any incomparable triple has a centroid.
Note that if a centroid exists, it is unique, as follows from the $B$-set axioms. It can also be checked that (B4) follows from (B5),
and in the finite case, (B4) and (B5) are equivalent (see \cite{Adeleke4} Lemma 29.1).

The relationship between $B$- and betweenness relations and `cycle-free' partial orders is explored in \cite{Truss}. If a $B$-set 
$(X, B)$ in addition satisfies $\forall x \forall y \forall z B(x; y, z) \vee B(y; z, x) \vee B(z; x, y)$ then is is called a 
{\em linear betweenness relation} (for any three points, one of them lies between the other two), and one can check that any linear 
$B$-relation arises from some linear order $\le$ on $X$ by saying that $B(x; y, z)$ if $y \le x \le z \vee z \le x \le y$. 

A {\em combinatorial tree} is a simple connected graph without circuits. Any combinatorial tree gives rise to a betweenness relation 
in an obvious way, derived from the notion of path. Thus $B(x; y, z)$ if the unique path from $y$ to $z$ passes through $x$. Conversely, 
any finite $B$-set of positive type on at least three points can be viewed as a combinatorial tree by saying that two points are adjacent 
if there is no vertex of their complement between them. See \cite{Adeleke4} page 103, where this is explained. To make things 
consistent, we also view a $B$-set on just two vertices as corresponding to a connected graph on two vertices (that is, just one 
edge). If there is just one vertex, then both the $B$-set and the corresponding graph are empty relations. From thinking of it as 
a graph, we may talk about the {\em valency} (or {\em degree}) of an element of a finite $B$-set, a {\em leaf} as a node of valency 
1, and a node of valency 2 as {\em dyadic}. 

Some of the definitions for semilinear orders carry over to betweenness relations, for instance that of `cone', here referred to as 
`branch'. If $a \in X$, there is an equivalence relation on $X \setminus \{a\}$ given by $y$ and $z$ are related if $\neg B(a; y, z)$ and 
the equivalence classes are called {\em branches} at $a$. If there are at least 3 distinct branches at $a$, then $a$ is called a 
{\em ramification point}. Note that if $X$ is finite, then the number of branches at $a$ is equal to its valency in the corresponding 
combinatorial tree. 

A {\em $C$-relation} is a ternary relation $C$ on a set $X$ satisfying the following properties:

(C1) $C(x; y, z) \to C(x; z, y)$,

(C2) $C(x; y, z) \to \neg C(y; x, z)$,

(C3) $C(x; y, z) \to C(x; w, z) \vee C(w; y, z)$,

(C4) $x \neq y \to C(x; y, y)$.

We also call $(X, C)$ a {\em $C$-set}. 

We have given the definitions of $B$- and $C$-sets as they both feature in the definition of `limit of betweenness relations'.

The definition given in \cite{Adeleke2} of a limit of betweenness relations defines what it means for a group to preserve such a 
limit, rather than saying what the limit actually is.

If $(G, \Omega)$ is a permutation group, following \cite{Adeleke2} Definition 2.1.9 and reproduced in \cite{Bhattacharjee} Definition 2.4, we say that $G$ {\em preserves a limit of betweenness relations} if it is an infinite Jordan group 
such that there are a linearly ordered set $(J, \le)$ with no least element, a chain $(\Gamma_j: j \in J)$ of subsets of $\Omega$, and 
a chain $(H_j: j \in J)$ of subgroups of $G$ such that $i < j \to \Gamma_i \supset \Gamma_j \wedge H_i \supset H_j$, and

(i) for each $j$, $H_j = G_{(\Omega \setminus \Gamma_j)}$, $H_j$ is transitive on $\Gamma_j$, and has a unique maximal congruence $\rho_j$ 
on $\Gamma_j$,

(ii) for each $j$, the group naturally induced by $H_j$ on $\Gamma_j/\rho_j$ is a 2-transitive but not 3-transitive Jordan group 
preserving a betweenness relation,

(iii) $\bigcup_{j \in J}\Gamma_j = \Omega$,

(iv) $(\bigcup_{j \in J}H_j, \Omega)$ is a 2-primitive but not 3-transitive Jordan group,

(v) if $i \ge j$ then $\rho_i \supseteq \rho_j \hspace{-.05in} \upharpoonright_{\Gamma_i}$,

(vi) $\bigcap_{j \in J}\rho_j$ is equality on $\Omega$, (where in order for the intersection to make sense we 
view $\Omega \setminus \Gamma_j$ as a single $\rho_j$-class),

(vii) $(\forall g \in G)(\exists i_0 \in J)(\forall i < i_0)(\exists j \in J)(g(\Gamma_i) = \Gamma_j \wedge gH_ig^{-1} = H_j)$,

(viii) for any $\alpha \in \Omega$, $G_\alpha$ preserves a $C$-relation on $\Omega \setminus \{\alpha\}$.

There are similar definitions of `limit of $D$-relations', and of `limit of Steiner systems', which we do not need.

\section{Trees of $B$-sets and $L$-relations}

A method was devised in \cite{Bhattacharjee} for building a limit of $B$-relations in a `generic' 
fashion, based on applying Fra{\"i}ss\'e's method to a class of finite structures, called `trees of 
$B$-sets'. The idea was that these would approximate the situation described in the definition of 
`limit of betweenness relations'. Since the method was fully generic, the tree was also allowed to expand, 
until in the limit it became an infinite dense semilinear order of positive type with all nodes having 
infinite ramification order, which is isomorphic to $K$, mentioned in the introduction. In addition, 
an associated first order structure $M$ was given, called an `$L$-structure', which corresponds 
precisely to the tree of $B$-sets, and such that the automorphism group of the Fra\"iss\'e limit of 
the tree of $B$-sets acts on $M$, and this action is the desired Jordan group. The construction of 
\cite{Adeleke1} which in some ways is superficially similar, is rather of the {\em group} than 
directly of a relation or relations preserved. Furthermore, it is constructed using a very different, 
and discrete semilinear order. One of our goals is to try to unify these two scenarios, and indeed, 
to render the Bhattacharjee-Macpherson approach considerably more general, with the object of finding 
a wider class of examples of limits of betweenness relations, and consequently, of Jordan groups, not 
hitherto constructed. In this section we explain the two main notions used, of trees of $B$-sets and 
$L$-structures, and elucidate the connection between them. The amalgamation arguments which allow us 
to deduce the existence of the `generic' structure, are given in the next section.

Let $T$ be an `ambient' meet-closed lower semilinear order, with meet given by $\curlywedge$ 
(generalizing the infinite semilinear orders arising in \cite{Bhattacharjee} and \cite{Bradley}). We 
usually refer to $T$ as a `tree' (in the partially ordered sense), and we also call $T$  the {\em 
structure tree} in what follows. Although there are many possibilities for what $T$ can be, they 
are nevertheless carefully circumscribed (in order to achieve the Jordan property in the structure 
it is used to help build). We shall insist that $T$ is countable with no least element, though we 
do allow it to have maximal elements, and it must ramify infinitely at each non-maximal node. In 
addition, its automorphism group must be `rich'. This is made more precise as follows.

Let $C$ be a fixed chain with no least element. We say that a tree $T$ is {\em $C$-coloured} if there 
is a `colouring function' $F: T \to C$ which is compatible with the ordering, meaning that 
$s < t \Rightarrow F(s) < F(t)$. A subset of $T$ which is a chain, and such that all members of $C$ 
arise as colours, is called a {\em $C$-chain}. Note that every $C$-chain is a maximal chain (though not 
every maximal chain need be a $C$-chain).  We remark that the following construction also works for 
finite ramification orders $\ge 2$, but in this paper we only require the infinitely branching case.

\begin{lemma} \label{3.1} For any countable chain $C$ with no minimal element, there is a countable 
$C$-coloured lower semilinearly ordered meet semilattice $T = T_C$ which is unique up to isomorphism 
subject to having the following  properties:  every non-maximal point of $T$ ramifies with ramification 
order $\aleph_0$, and every point of $T$ lies in a $C$-chain.  \end{lemma}

\begin{proof} Existence is established using the method of \cite{Droste}. We regard $C$ itself as one 
of the maximal chains of $T$. To get the other elements, we take $T$ to consist of all the finite sequences 
of the form $\sigma =  (c_0, (c_1, n_1), \ldots, (c_k, n_k))$ where $k, n_i \in \omega$, and 
$c_0 < c_1 < \ldots < c_k$ lie in $C$. This is partially ordered by saying that 
$\sigma = (c_0, (c_1, n_1), \ldots, (c_k, n_k))$ is less than $\tau = (d_0, (d_1, m_1), \ldots, (d_l, m_l))$ 
if $\sigma \neq \tau$, $k \le l$, $n_i = m_i$ for $0 < i \le k$, $c_i = d_i$ for $i < k$, and 
$c_k \le d_k$. Thus for instance, the sequences of length 1 are essentially the same as the members of $C$ 
under its given ordering. At any fixed $c \in C$, the sequences of length 2 beginning with $c$ branch off 
from $c$ in infinitely many copies of the half-chain $(c, \infty)$, and so on. In general. the points 
below $(c_0, (c_1, n_1), \ldots, (c_k, n_k))$ form a chain in order-type $(-\infty, c_k)$, and the points
of the form $(c_0, (c_1, n_1), \ldots, (c_k, n_k), (d,n))$ fall into $\aleph_0$ chains as $n$ varies, each
of order-type $(c_k, \infty)$. The effect is that $T$ is the union of many chains all isomorphic to $C$. 

To establish uniqueness, let $T_1$ and $T_2$ be two such trees, and we construct an isomorphism from $T_1$ 
to $T_2$ by back-and-forth. Let $\mathcal F$ be the family of all (colour-preserving) isomorphisms $f$ 
from a subset of $T_1$ to a subset of $T_2$, such that ${\rm dom}(f)$ and ${\rm range}(f)$ are finite 
unions of $C$-chains. We show that any member $f$ of $\mathcal F$ can be extended to include any given 
member $s$ of $T_1$ in its domain (and similarly it can be extended to include any member of $T_2$ in its 
range). Since $T_1$ and $T_2$ are countable, we can enumerate their members and alternately extend in $T_1$ 
and $T_2$, so that after countably many steps, the union of all the maps is an isomorphism from $T_1$ to $T_2$. 

If $s \in {\rm dom}(f)$, no extension is necessary. Otherwise  choose points in the finitely many chains 
forming ${\rm dom}(f)$ having the same colour as $s$. Since $T_1$ is a meet semilattice, we can form 
the meets of each of these elements with $s$, which must all be below $s$, so are linearly ordered. 
Let the greatest of these be $t$. This is the ramification point at which $(-\infty, s]$ `branches off' from 
${\rm dom}(f)$. By assumption, $s$ lies in a $C$-chain $X$ of $T_1$. Also $t \in {\rm dom}(f)$ since it 
lies in one of the chains which make up ${\rm dom}(f)$, and $t$ and $f(t)$ have the same colour. As
$f(t)$ has infinite ramification order in $T_2$, there is a $C$-chain $Y$ of $T_2$ which intersects 
${\rm range}(f)$ in $(-\infty, f(t)]$. The two chains $X$ and $Y$ restricted to colours greater than $F(t)$
are isomorphic (the isomorphism being provided by colours) and this provides a unique extension of $f$  
from ${\rm dom}(f) \cup X$ to ${\rm range}(f) \cup Y$.           \end{proof}

The same back-and-forth idea can be used to show that ${\rm Aut}(T)$ acts transitively on the 
family of $C$-chains of $T$, explaining the sense in which we regard its automorphism group as `rich'.  
Note however that the $C$-chains we have explicitly included are not the only $C$-chains, as many others 
can be obtained by diagonalizing. These may also not be the only maximal chains of $T_C$.  For instance 
for the ambient tree of \cite{Bhattacharjee}, which may be obtained by this method  by taking $C = {\mathbb Q}$, 
there will be many maximal chains which peter out well below the top at irrational points. At 
any rate, the key point is that all points of $T_C$ thus constructed correspond, via their 
final entries, to points of $C$, which we may think of as their `level', and they are coloured accordingly 
(which is our real reason for using the letter `C'). We further note that although the ambient tree of 
\cite{Bhattacharjee} is as stated, its automorphism group is much bigger, as there there is no requirement 
to preserve colours. 

We augment the language of partial orders by means of unary predicates for the members of $C$. Our 
structures then may be regarded as `coloured trees', where each node of the tree given in the first paragraph 
of the proof of Lemma \ref{3.1} is coloured by the greatest member of $C$ which appears in one of its entries. 
In \cite{Bradley}, one particular tree termed the ${\mathbb N}^+$-tree was used, and the language used 
to describe this had `depth' predicates, to tell us how far from the top any particular point was. In 
that case, every point was at a finite distance from the top. In the general setting this will not be true, which
is the reason for our more complicated description of the colouring. 

The definition of `tree of $B$-sets' is quite complicated, so we begin with the finite case. A {\em finite 
tree of $B$-sets} $A$ consists of a finite meet-closed (coloured) subtree $(T^A, \curlywedge, <)$ of $T_C$, 
with root $r$ and with all non-leaf vertices ramifying, and for each $t \in T^A$ a finite $B$-relation $B_t$ 
of positive type having domain $B(t)$, together with maps $f_t$ and $g_{st}$ for $s < t$ in $T^A$ such that 
$t$ is a successor of $s$, fulfilling various conditions spelt out below. We say that $T^A$ is the subtree 
of $T_C$ which is {\em populated} by $A$.

In the preamble, there were two related notions, of {\em cones} of points in a semilinear order, and 
of {\em branches} of a $B$-relation. Both of these feature in our definition, which is why they need to 
be distinguished. As far as cones are concerned, since $T^A$ is finite, the cones at $t$ are essentially 
the same as its successors; more precisely, the set of successors $succ(t)$ of $t$ in $T^A$ forms a 
natural family of representatives of the cones at $t$ (their minimal elements). For $t$ such that $B_t$ is 
not a linear $B$-relation we also postulate the existence of a surjection $f_t$ from $\{s: t < s\}$ to the 
set of ramification points of $B(t)$, such that $f_t(s_1) = f_t(s_2)$ if and only if $s_1$ and $s_2$ lie in 
the same cone (which is equivalent to saying that they are above the same successor of $t$). This clearly 
entails that the number of ramification points of the $B$-set $B(t)$ is equal to the ramification order 
of $t$ in $T^A$, and since we required that all non-leaf vertices of $T^A$ ramify, $t$ is a leaf if and 
only if $B_t$ has no ramification points, which is equivalent to saying that it is a linear betweenness 
relation. We may write $f^{-1}_t(x)$ for the member of $succ(t)$ mapped by $f_t$ to $x$. In addition, 
the relations between the $B$-sets at comparable vertices of $T^A$ will be controlled by a family of 
functions $\{g_{st}\}_{t \in succ(s)}$. If $t$ is a successor of $s$, so that $f_s(t)$ is a 
ramification point of $B(s)$, then $g_{st}$ is a surjection from $B(s) \setminus \{f_s(t)\}$ to $B(t)$ 
which is constant on each branch of $B_s$ at $f_s(t)$, and which induces a bijection from the set of 
branches of $B_s$ at $f_s(t)$ to the vertices of $B(t)$. This in particular requires that $|B(t)|$ 
equals the number of branches at $f_s(t)$. 

In summary, a {\em finite coloured tree of $B$-sets} $A$ is a finite meet-closed subtree $T^A$ of $T_C$, with 
colours inherited from $T_C$, and with the $B$-sets and functions $f_t$ and $g_{st}$ fulfilling the stated conditions. 

We may extend the definition of $g_{st}$ to the case of general $s \le t$. Let $g_{s s}$ be the identity on 
$B(s)$. If $s < t$, choose a (necessarily unique) maximal chain $s = s_0 < s_1 < \ldots < s_n = t$, and let 
$g_{st}$ be the composite map $g_{s_{n-1} s_n}g_{s_{n-2} s_{n-1}} \ldots g_{s_0 s_1}$, whose domain is contained in $B(s)$. 

\begin{lemma} \label{3.2} If $A$ is a finite tree of $B$-sets, and $s < t < u$ in $T^A$, then 
$g_{su} = g_{tu} g_{st}$, and each $g^{-1}_{su}(a)$ is a union of branches around $f_s(u) \in B(s)$. \end{lemma}

\begin{proof} Write the unique maximal chains from $s$ to $t$ as $s_0 < s_1 < \ldots < s_m$ and from $t$ to $u$ 
as $s_m < s_{m+1} < \ldots < s_n$. Then both $g_{su}$ and $g_{tu} g_{st}$ are equal to the composition 
$(g_{s_{n-1} s_n}g_{s_{n-2} s_{n-1}} \ldots g_{s_m s_{m+1}})(g_{s_{m-1} s_m}g_{s_{m-2} s_{m-1}} \ldots g_{s_0 s_1})$.  
\end{proof}

We observe that we are following the presentation of \cite{Bhattacharjee} rather than \cite{Bradley}, in order to 
facilitate the proof of amalgamation. Thus the map $f_t$ has range consisting of just the set of ramification 
points of $B(t)$, and a leaf of the tree is labelled by any (finite) path (linear betweenness relation). Note
however that our maps $g_{rt}$ are the inverses of the ones used in \cite{Bhattacharjee}.

Now the limit structure will be an infinite tree of $B$-sets, and the Jordan group will be its 
automorphism group. In order to verify the Jordan property, we have to say what set it acts on as a 
group of permutations, and there will be a suitable set, together with a ternary relation $L$ on that 
set fulfilling this role.  In \cite{Bradley} and \cite{Bhattacharjee} the relationship between a 
finite tree of $B$-sets and an $L$-structure was established. The amalgamation argument was phrased 
in terms of the $L$-structure, to make it easier to appeal to a version of Fra\"iss\'e's Theorem 
available in the literature for first order structures, and then afterwards the tree of $B$-sets 
structure was recovered from the (now infinite) $L$-relation. Our choice here is rather to treat 
both of these aspects, tree of $B$-sets, and $L$-relation, simultaneously, which involves proving 
the Fra\"iss\'e style result directly, but it does mean that the final deduction of the existence 
of the infinite tree of $B$-sets is eased. So now we have to introduce the general notions of 
`tree of $B$-sets' and $L$-relation.

A \emph{tree of $B$-sets} comprises a semilinear order $(T,<)$ (which may be coloured), which is a lower semi-lattice, 
and for each node $t$, a $B$-set $(B(t), B_t)$ of positive type, together with functions $f_t$ and $g_{s t}$
for $s \le  t$ in $T$, satisfying the following properties:

%if $B_t$ is not a linear betweenness relation, then 
$f_t$ is a surjection from $\{s \in T: t < s\}$ to the set of ramification points of $B(t)$ such that for all $s_1, s_2 > t$, we have $f_t(s_1) = f_t(s_2)$ if and only if $s_1,s_2$ lie in the same cone above $t$;

$g_{s s}$ is the identity on $B(s)$; 

if $s < t$, $g_{s t}$ is a surjection from $B(s) \setminus \bigcup_{s \le u < t}g^{-1}_{s u} (f_u(t))$ to $B(t)$, and $g_{s t}(x) = g_{s t}(y)$ 
if and only if for some $u$ with $s \le u < t$, $g_{s u}(x)$ and $g_{s u}(y)$ lie in the same branch 
of $B_u$ at $f_u(t)$;

and if $s \le t \le u$, then $g_{s u} = g_{t u} \circ g_{s t}$.

We remark that this directly generalizes the definition above of finite tree of $B$-sets. The main case to 
note is that if $t$ is a successor of $s$, then $s \le u < t \Rightarrow s = u$, so the union reduces
to just one member, namely $g^{-1}_{s s}(f_u(t)) = f_u(t)$, so this agrees with the earlier condition. Note also that these conditions imply $B(t)$ is a linear betweenness relation if and only if $t$ is a maximal element of $T$.

\vspace{.1in}

Associated with any finite tree $A$ of $B$-sets the `$L$-set', which is a finite set $M^A$ on which there
is a ternary relation $L = L^A$, is now defined in what follows. In order to make the notation work well, 
we shall ensure that the members of the $B$-sets occurring in $A$ will be represented by finite subsets 
of $M^A$. This could be achieved formally by replacing the whole tree of $B$-sets by an isomorphic copy.
We choose not to make this too explicit, and instead blur the distinction between a member of one of the 
$B$-sets and the finite subset of $M^A$ which corresponds to it. The notation becomes particularly useful 
when we describe how the finite structures approximate the limit.

The domain of $M^A$ is taken to be that of the $B$-set at the root of $T^A$, namely $B(r)$, and the 
members of $B(r)$ may now be viewed as singletons of this domain according to the identification just
mentioned. More generally, if $a \in B(r)$ and $t \in T^A$, we let $[a]_t = g_{rt}^{-1}g_{rt}(a)$, provided 
that $a \in {\rm dom}(g_{rt})$. In other words, $[a]_t$ is the set of members of $B(r)$ which are mapped 
by $g_{rt}$ to the same point as $a$. By surjectivity of $g_{rt}$, all vertices of $B(t)$ can be written 
in the form $g_{rt}(a)$ for some $a \in B(r)$. This $a$ will  not be unique, if $r \neq t$. The $[a]_t$ 
for $a \in {\rm dom}(g_{r t})$ are therefore pairwise disjoint non-empty subsets of $M^A$. Note that if 
$r = t$, $[a]_r$ agrees with our decision to regard a member of $B(r)$ as the singleton $\{a\}$. Using 
this notation, the maps $g_{st}$ are instead given by $g_{st}[a]_s = [a]_t$.

We take a moment to justify this more formally. The original finite tree of $B$-sets has a tree $T^A$,
each $(B(t), B_t)$ is a $B$-set, and maps $f_t$ and $g_{st}$ for $s \le t$ are given. The isomorphic copy 
then has the same tree $T^A$, and $B'(t) = \{[a]_t: a \in {\rm dom } \, g_{rt}\}$. The bijection $\varphi_t$ 
from $B(t)$ to $B'(t)$ is given by $\varphi_t(g_{rt}(a)) = [a]_t$, and this induces a $B$-relation $B_t'$ 
on $B'(t)$. Finally we define $f_s'(t) = \varphi_s(f_s(t))$ if $s < t$, and $g'_{st}[a]_s = [a]_t$. It is
easily checked that this is an `arboreal isomorphism' (see below for the definition), the key point being
that $g'_{st}\varphi_s(x) = g'_{st}\varphi_s g_{rs}(a)$ (for some $a$), 
$= g'_{st}[a]_s = [a]_t = \varphi_t(g_{rt}(a)) = \varphi_t(g_{st}(x))$.

The ternary relation $L$ on $M^A$ is given by saying that $L(a; b, c)$ holds for distinct $a$, $b$, $c$,
provided that for some $t$, $g_{rt}(a)$, $g_{rt}(b)$, and $g_{rt}(c)$ are defined and distinct, and 
$B_t(g_{rt}(a); g_{rt}(b), g_{rt}(c))$. In the notation just introduced, this just says that
$[a]_t$, $[b]_t$, and $[c]_t$ are defined and distinct, and $B_t([a]_t; [b]_t, [c]_t)$.

A key point here is to show that we can pass freely between the first order structure $(M^A, L)$ (the 
`$L$-set') and the tree of $B$-sets from which it was defined.

\begin{lemma} \label{3.3} If $A$ is a finite tree of $B$-sets, and $L$ is the ternary relation on $M^A$ 
as defined above, then for any $a, b, c \in M^A$ for which $L(a; b, c)$, there is a unique $t \in T^A$ 
such that $[a]_t, [b]_t$, and $[c]_t$ are defined and distinct, and $B_t([a]_t; [b]_t, [c]_t)$. 
Furthermore, for any distinct $a, b, c \in M^A$, the $L$-relation holds between $a, b,$ and $c$ in some 
order. \end{lemma}

\begin{proof} First note that if $s \le t$ and $[a]_t$ is defined, then so is $[a]_s$, since 
$g_{rt} = g_{st} \circ g_{rs}$. Furthermore, if two of $[a]_s$, $[b]_s$, and $[c]_s$ are (defined and) 
equal, say $[a]_s = [b]_s$, then also $[a]_t = [b]_t$ (meaning that if one of them is defined, then so 
is the other, and they are equal), again since $g_{rt} = g_{st} \circ g_{rs}$. Hence the set $S$ of nodes 
$s$ at which $[a]_s$, $[b]_s$, and $[c]_s$ are defined and distinct is downwards closed in $T$, and by 
hypothesis contains the root $r$. We show that $S$ is linearly ordered, and its maximal element $s$ 
gives the desired node. For this we see that if $s \in S$ and $\{[a]_s, [b]_s, [c]_s\}$ is 
incomparable in $B_s$, then there is a unique node of $S$ which is a successor to $s$. For as 
$B_s$ is a $B$-relation of positive type, there is a unique centroid of these three points, which may be 
written as $[d]_s$, and they lie in distinct branches of $[d]_s$, so if $t$ is the successor of $s$ for 
which $f_s(t) = [d]_s$ then $[a]_t$, $[b]_t$, and $[c]_t$ are defined and distinct, so $t \in S$. If 
$[e]_s$ is any other point of $B_s$, then two of $[a]_s$, $[b]_s$, and $[c]_s$ lie in the same branch at 
$[e]_s$, so if $u$ is the successor of $s$ for which $f_s(u) = [e]_s$ (if it exists at all), then 
$[a]_u$, $[b]_u$, and $[c]_u$ are not distinct, and $u \not \in S$. 

Since $S$ is finite, this process must terminate, and the only way that this can happen is that one of 
$[a]_s$, $[b]_s$, and $[c]_s$ lies between the other two in $B_s$. To see that this $s$ is maximal 
in $S$, suppose otherwise, and let $t \in succ(s)$ lie in $S$. Then $[a]_t$, $[b]_t$, and $[c]_t$ all 
exist, so $f_s(t) = [d]_s$ for some $[d]_s \neq [a]_s, [b]_s, [c]_s$. Then one sees that two of 
$[a]_s$, $[b]_s$, and $[c]_s$ lie in the same branch at $[d]_s$, so that actually their images under 
$g_{st}$ are not distinct after all. Thus $s$ is the unique greatest point of $S$, and since 
$L(a; b, c)$ it must be $[a]_s$ which lies between the other two. 

For the final remark, the same argument shows that there is a greatest node $s$ at which $[a]_s$, 
$[b]_s$, and $[c]_s$ are defined and distinct. By maximality, one of them must lie between the other 
two in $B_s$, which tells us that the $L$-relation holds for $a, b,$ and $c$ in some order.  \end{proof}

In summary, if $L(a; b,c)$ holds then there is a unique $s$ at which $B_s([a]_s; [b]_s, [c]_s)$ holds 
for distinct $[a]_s$, $[b]_s$, $[c]_s$. We say that $L(a; b, c)$ is {\em witnessed} in $(B(s), B_s)$. 

\begin{lemma} \label{3.4} If $B_1$ and $B_2$ are $B$-relations of positive type on the same set 
$M$, and $B_1 \subseteq B_2$, then $B_1 = B_2$.   \end{lemma}

\begin{proof} If not, $B_1$ is a proper subset of $B_2$, so there must be a triple such that $B_2(a;b,c)$ 
but not $B_1(a;b,c)$. Consequently $a$, $b$, $c$ cannot be $B_1$-related in any order. First they 
must be distinct, for if $a=b$ or $c$, then $B_1(a;b,c)$ holds. If $b=c$, then from $B_2(a;b,c)$ it 
follows that $a = b = c$. Now since all three are distinct, and if for instance $B_1(b;a,c)$, then also 
$B_2(b;a,c)$ from which by axiom (B2), $a = b$ after all, and similarly in other cases. Since 
$\{a, b, c\}$ is $B_1$-incomparable, and $B_1$ is of positive type, there is $x \neq a, b, c$ such that 
$B_1(x;a,b)$ and $B_1(x;a,c)$ and $B_1(x;b,c)$. Since $B_2(a;b,c)$, by axiom (B3), $B_2(a;b,x)$ or 
$B_2(a;x,c)$. Hence (as $B_1 \subseteq B_2$) either $B_2(x;a,b)$ and $B_2(a;x,b)$ which implies that $x = a$, 
or else $B_2(x; a,c)$ and $B_2(a;x,c)$, which also implies that $x = a$, contradiction.    \end{proof}

\begin{lemma} \label{3.5} If $A$ is a finite tree of $B$-sets, and $L$ is the ternary relation on 
$M^A$ as defined above, then $B_r$ is the unique $B$-relation of positive type on $B(r)$ whose irreflexive 
version is contained in $L$.   \end{lemma}

\begin{proof} By the `irreflexive version' $irr(B)$ of $B$ we understand $B \setminus \{(a; b,c): a = b \vee a = c\}$. 
(Note that we don't need to remove triples $(a; b, c)$ such that $b = c$ since by the 
axioms for $B$-relations, any such in $B$ would also satisfy $a = b = c$.) 

Let $B$ be a $B$-relation of positive type on $B(r)$, and suppose that $irr(B) \subseteq L$. Suppose for a 
contradiction that $B \not \subseteq B_r$, and let a triple $(a;b,c)$ lie in $B$ but not in $B_r$. If 
two of $a, b, c$ were equal, then $(a; b,c)$ would also lie in $B_r$. Hence they are all distinct, so 
$(a; b, c)$ lies in $irr(B)$, and hence also in $L$. By Lemma \ref{3.3} there is a unique $t \in T$ at which 
$L(a;b,c)$ is witnessed. In other words, $[a]_t$, $[b]_t$, and $[c]_t$ are defined and distinct 
and $B_t([a]_t; [b]_t, [c]_t)$. Since $\neg B_r(a; b, c)$, we see that $r < t$. Let $s$ be the successor 
of $r$ such that $r < s \le t$. Then $a$, $b$, and $c$ lie in distinct branches of the ramification 
point $w = f_r(s)$. Since $B$ is a $B$-relation, $B(a; b, w)$ or $B(a; w, c)$. Since $a$, $b$, $c$ 
lie in distinct branches of $w$, $B_r(w; a, b)$, and $B_r(w; a, c)$, and hence $L(w; a, b)$ and 
$L(w; a, c)$. Therefore from the properties of $L$-relations, $\neg L(a; w, b)$ and $\neg L(a; w, c)$. 
Therefore, as $irr(B) \subseteq L$, $\neg B(a; w, b)$ and $\neg B(a; w, c)$. This gives a contradiction.
      
We deduce that $B \subseteq B_r$, and hence by Lemma \ref{3.4}, $B = B_r$.       \end{proof}

To make the inductive part of Lemma \ref{3.7} work smoothly, it is helpful to be able to restrict a tree 
of $B$-sets to the set of vertices above some fixed vertex. If $A$ is a finite tree of $B$-sets, and 
$s \in T^A$, let $A^{\ge s}$ be the tree of $B$-sets having vertices in $\{t \in T^A: s \le t\}$ and with 
$B$-sets and $f$ and $g$ functions inherited from those of $A$. It is clear that this is a finite 
tree of $B$-sets.

\begin{lemma} \label{3.6} Let $A$ be a finite tree of $B$-sets, and $s \in T^A$. Let $L$ be the ternary 
relation on $M^A$ as defined above. Then the $L$-relation on the restricted tree of $B$-sets $A^{\ge s}$ 
is equal to the set of triples of the form $(g_{rs}(a); g_{rs}(b), g_{rs}(c))$ such that $(a; b, c) \in L$. \end{lemma}

\begin{proof} Applying the definitions to $A^{\ge s}$, we see that every member of $B(t)$ for $t \ge s$ may 
be written in the form $[[a]_s]_t$ for some $[a]_s \in B(s)$ (in the domain of $g_{st}$). Now
$[[a]_s]_t = g_{st}[a]_s = g_{st}g_{rs}[a]_r = g_{rt}[a]_r = [a]_t$. Thus $([a]_s; [b]_s, [c]_s)$ lies in
the $L$-relation as defined in $A^{\ge s}$ if there is $t \ge s$ such that $[[a]_s]_t$, $[[b]_s]_t$, and 
$[[c]_s]_t$ are defined and distinct, and $([[a]_s]_t; [[b]_s]_t, [[c]_s]_t)$ lies in $B_t$. This is 
therefore the same as requiring that $[a]_t$, $[b]_t$, and $[c]_t$ are defined and distinct, and 
$([a]_t; [b]_t, [c]_t) \in B_t$, which is the same as saying that $(a; b, c) \in L$.  \end{proof}

We now work in a first order language $\mathcal L$, which has just $L$ as relation symbol. (In \cite{Bradley} 
a two-sorted language was used at this point, and there, and also in \cite{Bhattacharjee}, a 
quaternary relation was added. We shall see that this is unnecessary.)  The notion of isomorphism for 
$L$-sets is as usual for first order structures.

For the trees of $B$-sets, the natural notion of isomorphism is called `arboreal'. An {\em arboreal isomorphism} 
between trees of $B$-sets $A$ and $A'$ over the same structure tree $T$ is a tree automorphism $\tau$ from 
$T^A$ to $T^{A'}$, together with a family of $B$-set isomorphisms $\varphi_s: B(s) \to B(\tau(s))$ for each 
$s \in T^A$, which respect the $f$ and $g$ maps, meaning that for each $s < t$ in $T^A$, and 
$x \in {\rm dom}(g_{st})$, $f_{\tau(s)} \tau(t) = \varphi_s f_s(t)$, and 
$g_{\tau(s)\tau(t)} \varphi_s(x) = \varphi_t g_{st}(x)$. In the finite case, it is sufficient to require 
this just when $t$ is a successor of $s$ (and the general statement follows by composition).

There is a rather stronger notion, also required. An {\em inner arboreal isomorphism} is an 
arboreal isomorphism for which $\tau$ is the restriction of an automorphism of the ambient tree 
$T$, now augmented by colours from $C$ (so the automorphisms are required to preserve `levels'). Note that 
this is definitely a stronger notion than just saying that it takes $T^A$ 
isomorphically to $T^{A'}$, since the way in which $T^A$ and $T^{A'}$ are situated inside $T$ needs 
to be taken into account. In \cite{Bhattacharjee} this was unnecessary, since any isomorphism from 
$T^A$ to $T^{A'}$ would extend to an automorphism (since these are meet-closed subtrees), but in the 
discrete case this is false, since levels (colours) are required to be preserved.

We note that in \cite{Bhattacharjee}, the trees $T^A$ are not taken explicitly as subtrees of a given tree; 
rather, as the Fra\"iss\'e limit is constructed during the proof, the tree emerges as the generic tree 
(rational maximal chains, all vertices ramifying with ramification order $\aleph_0$). By contrast, in 
\cite{Bradley}, the tree is given in advance, in that case having discrete maximal chains, and this of course 
affects how the extensions can be carried out. Our philosophy here is to combine these two approaches, 
by insisting that the tree $T$ is specified in advance, and the $T^A$ are subtrees, but the capacity to 
perform extensions, and therefore how the lemmas unfold which are required to make everything work, 
depends very much on what $T$ actually is. Note that in \cite{Bhattacharjee}, by `structure tree' is meant
the {\em finite} tree on which the finite tree of $B$-sets is defined, whereas for us, it means the 
(infinite, `ambient') tree $T$, whose finite subtrees give rise to the finite trees of $B$-sets.

\begin{lemma} \label{3.7} If $A$ and $A'$ are finite trees of $B$-sets over the same structure tree 
$T$, for any arboreal isomorphism from $A$ to $A'$ there is an $\mathcal L$-isomorphism between 
the corresponding structures on $M^A$ and $M^{A'}$. 

Conversely, any isomorphism between $\mathcal L$-structures associated with finite trees of $B$-sets 
$A$ and $A'$ arises from an arboreal isomorphism from $A$ to $A'$. \end{lemma}

\begin{proof} First suppose that $\tau$ and $B$-set isomorphisms $\varphi_s$ for $s \in T^A$ witness 
the given arboreal isomorphism from $A$ to $A'$. Since we are representing $B(r)$ as singletons of
members of $M^A$, the domain of $M^A$ is equal to $\{a: [a]_r \in B(r)\}$ where $(B(r), B_r)$ is the 
$B$-set at the root $r$. We are given the $B$-set isomorphism $\varphi_r: B(r) \to B(\tau(r))$, and 
we may define $\psi$ on $M^A$ by letting $\{\psi(a)\} = \varphi_r([a]_r)$. We have to see that this 
respects $L$. Suppose then that $L(a; b, c)$ holds in $M^A$. Then there is some (unique, by Lemma 
\ref{3.3}) $s$ such that $[a]_s$, $[b]_s$, and $[c]_s$ are defined and distinct, and 
$B_s([a]_s; [b]_s, [c]_s)$ holds in $A$. Since $\varphi_s$ is a $B$-set isomorphism, 
$B_{\tau(s)}(\varphi_s([a]_s); \varphi_s([b]_s), \varphi_s([c]_s))$ holds in $A'$. Thus, 
$\varphi_s([a]_s)$, $\varphi_s([b]_s)$, $\varphi_s([c]_s)$ are defined and distinct, and so 
$L(\varphi_s([a]_s); \varphi_s([b]_s), \varphi_s([c]_s))$ holds in $M^{A'}$, as required.

Conversely, suppose that $\psi$ is an isomorphism between the two $\mathcal L$-structures $L$ and 
$L'$ on $M^A$ and $M^{A'}$ arising from finite trees of $B$-sets $A$ and $A'$. We construct the 
isomorphism $\tau$ from the tree $T^A$ to $T^{A'}$, and $B$-set isomorphisms $\varphi_s$ for
$s \in T^A$ inductively, starting at the roots $r$ and $r'$. We let $\tau(r) = r'$. Now 
$B(r) = \{[a]_r: a \in M^A\}$ and $B_r$ is a $B$-relation of positive type on $M^A$, whose 
irreflexive part is contained in $L$. Hence $\psi(B_r)$ is a $B$-relation of positive type on 
$M^{A'}$ whose irreflexive part is contained in $L'$. By Lemma \ref{3.5} there is a unique such, 
which is $B_{r'}$. Hence $\psi(B_r) = B_{r'}$, and so we can let $\varphi_r([a]_r) = [\psi(a)]_{r'}$.

We now define $\tau(s)$ and $\varphi_s$ inductively on the height of $s$ in $T^A$ (essentially as 
in \cite{Bhattacharjee} with minor alterations), the basis case having just been done. Suppose that 
$\tau(s)$ and $\varphi_s: B(s) \to B(\tau(s))$ have been defined, and we show how to define $\tau(t)$ 
and $\varphi_t$ for each successor $t$ of $s$. Now $f_s(t) \in B(s)$, so $\varphi_s(f_s(t)) \in B(\tau(s))$, 
which is a ramification point, hence in the range of $f_{\tau(s)}$. Furthermore, $f_s$ is a bijection 
from $succ(s)$ to the set of ramification points of $B(s)$, and similarly, $f_{\tau(s)}$ is a bijection 
from $succ(\tau(s))$ to the set of ramification points of $B(\tau(s))$. We can therefore let 
$\tau(t) = f^{-1}_{\tau(s)}\varphi_s f_s(t)$, and this extension of $\tau$ to $succ(s)$ is a bijection 
to $succ(\tau(s))$. Furthermore, $f_{\tau(s)}\tau = \varphi_sf_s$ as required in the definition of 
`arboreal isomorphism'.

If $t \in succ(s)$, then $g_{st}$ maps $B(s) \setminus \{f_s(t)\}$ onto $B(t)$, and constitutes a 
bijection between the branches of $B_s$ at $f_s(t)$ and the points of $B(t)$, and similarly, 
$g_{\tau(s)\tau(t)}$ gives a bijection between the branches of $B_{\tau(s)}$ at $f_{\tau(s)}(\tau(t))$ 
and the points of $B(\tau(t))$. We may therefore let 
$\varphi_t = g_{\tau(s)\tau(t)} \varphi_s g_{st}^{-1}$. Note that here for each $x \in B(t)$, $g_{st}^{-1}(x)$ 
is a branch of $B_s$ at $f_s(t)$, and as $\varphi_s$ is a $B$-set isomorphism,
$\varphi_s g_{st}^{-1}(x)$ is a branch of $B_{\tau(s)}$ at $f_{\tau(s)}(\tau(t))$, which therefore gets 
mapped to a single point by $g_{\tau(s)\tau(t)}$. Since $\varphi_t$ is a bijection from $B(t)$ to 
$B(\tau(t))$, the image $\varphi_t(B_t)$ of the $B$-relation $B_t$ under $\varphi_t$ is a $B$-relation 
$B_{\tau(t)}$ of positive type on $B(\tau(t))$. We show that its irreflexive part is contained in the restriction of 
$L'$ to the points above $\tau(t)$. For let $[a]_t$, $[b]_t$, $[c]_t$ be distinct members of $B(t)$ such 
that $B_t([a]_t; [b]_t, [c]_t)$. By definition, it follows that $L(a; b, c)$, and since $\psi$
preserves the $L$-structure, $L'(\psi(a); \psi(b), \psi(c))$. By definition of $\varphi_r$, 
$L'(\varphi_r([a]_r); \varphi_r([b]_r), \varphi_r([c]_r))$. By Lemma \ref{3.6}, this is equivalent to 
$L'([\varphi_t(a)]_{\tau(t)}; [\varphi_t(b)]_{\tau(t)}, [\varphi_t(c)]_{\tau(t)})$. By Lemma \ref{3.5} 
this is equal to $irr(B_{\tau(t)})$. Finally, $g_{\tau(s)\tau(t)} \varphi_s = \varphi_t g_{st}$ as 
required.      \end{proof} 

In order to form our intended limit, which will give rise to the desired Jordan group preserving a 
limit of $B$-relations, we shall use a Fra\"iss\'e-style method, and the key point is to establish 
amalgamation for finite trees of $B$-sets. Up till now we have just considered trees of $B$-sets up 
to isomorphism. For the Fra\"iss\'e theory, it is crucial however to have the correct notion of 
substructure. The care that has to be taken is illustrated by an example given in \cite{Bhattacharjee} 
pages 66,67. That showed that additional relations needed to be added to make things work correctly. 
We can however achieve the same effect by carefully formulating the notion of substructure.

Let us say that if $B_1$ and $B_2$ are $B$-sets of positive type on sets $X_1$ and $X_2$, then $B_1$ is 
a {\em strong substructure} of $B_2$ if $X_1 \subseteq X_2$, and $B_1 \subseteq B_2$, and $B_1$ equals
the restriction $B_2 \upharpoonright X_1$ of $B_2$ to $X_1$. One checks that this is equivalent to saying
that for any incomparable (in either $B_1$ or $B_2$) $x, y, z \in X_1$ their centroid is the same in 
$B_1$ and $B_2$ (and hence lies in $X_1$). This gives rise to a similar notion of {\em strong embedding} 
of $B$-sets of positive type in the obvious way. There is a corresponding notion of strong substructure 
of a tree of $B$-sets. Because of complications about identification of points in different structures, 
we rather give a definition of `strong (coloured) arboreal embedding' (similarly for strong inner arboreal 
embedding). Though this is mainly used in the finite case, it is required later on for possibly infinite 
trees of $B$-sets, so we give the general definition. A {\em strong arboreal embedding} of trees of 
$B$-sets $A_1$ into $A_2$ over the same ambient tree $T$ consists of a tree embedding $\tau$ from 
$T^{A_1}$ to  $T^{A_2}$ (which means that it preserves colours, $<$ and meets), together with a family of 
strong $B$-set embeddings $\varphi_t: B(t) \to B(\tau(t))$ for each $t \in T^{A_1}$, which respect the 
$f$ and $g$ maps, meaning that for each $t \in T^{A_1}$, $\varphi_tf_t = f_{\tau(t)}\tau$, and 
for each $s \le  t$ in $T^{A_1}$, $g_{\tau(s)\tau(t)} \varphi_s = \varphi_t g_{st}$. We say that $A_1$
is a {\em strong substructure} of $A_2$ if $A_1$ is a substructure, and the inclusion map is a strong 
embedding, where by the `inclusion map' we mean $(\tau, \{\varphi_t: t \in T^{A_1}\})$ where $\tau$ is
inclusion from $T^{A_1}$ to $T^{A_2}$ and for each $t \in T^{A_1}$, $\varphi_t$ is inclusion from 
$(B(t), B_t)$ in $A_1$ to $(B(t), B_t)$ in $A_2$.

The following result is adapted from Lemma \ref{3.7} for embeddings in place of isomorphisms.

\begin{lemma} \label{3.8} If $A_1$ and $A_2$ are trees of $B$-sets over the same structure tree 
$T$, for any strong embedding  from $A_1$ to $A_2$ there is an embedding between the 
corresponding $L$-structures on $M^{A_1}$ and $M^{A_2}$.    \end{lemma}

\begin{proof} Suppose that we are given a strong embedding of $A_1$ into $A_2$. The embedding is 
presented by means of a tree-embedding $\tau$ of $T^{A_1}$ into $T^{A_2}$ and a family of 
strong $B$-set embeddings $\varphi_s$ from $(B(s), B_s)$ in $A_1$ to $(B(\tau(s)), B_{\tau(s)})$ in $A_2$.

The domains of $M^{A_1}$ and $M^{A_2}$ are the $B$-sets $B(r_1)$ and $B(r_2)$ at the roots $r_1$ and 
$r_2$ of $T^{A_1}$, $T^{A_2}$ respectively. For each $a \in B(r_1)$, let $\psi(a)$ be some member of 
$g_{r_2 \tau(r_1)}^{-1}\varphi_{r_1}(a)$. We observe that $\varphi_{r_1}$ maps $B(r_1)$ to $B(\tau(r_1))$, 
and since $r_2$ is the root of $T^{A_2}$, we obtain $r_2 \le \tau(r_1)$. As $g_{r_2 \tau(r_1)}$ maps a subset of 
$B(r_2)$ surjectively to $B(\tau(r_1))$, such $\psi(a)$ exists, and lies in $B(r_2)$. Since 
$\varphi_{r_1}$ is injective, so is $\psi$, and it maps $B(r_1)$ into $B(r_2)$.

It remains to verify that if $L_1$ and $L_2$ are the $L$-relations on $M^{A_1}$ and $M^{A_2}$
respectively, then $L_1(a; b, c) \Leftrightarrow L_2(\psi(a); \psi(b), \psi(c))$, where $a, b, c$ are 
distinct points of $B(r_1)$. Now if  $L_1(a; b, c)$ holds, then by Lemma \ref{3.3} there is a unique 
$s \in T^{A_1}$ such that $[a]_s$, $[b]_s$, and $[c]_s$ are defined and distinct, and 
$B_s([a]_s; [b]_s, [c]_s)$. Since $\varphi_s$ is a strong $B$-set embedding,  
$\varphi_s([a]_s)$, $\varphi_s([b]_s)$, and $\varphi_s([c]_s)$ are distinct, and 
$B_{\tau(s)}(\varphi_s([a]_s); \varphi_s([b]_s), \varphi_s([c]_s))$. By definition of `embedding', 
$\varphi_s([a]_s) = [\varphi_{r_1}(a)]_{\tau(s)}$, and by choice of $\psi$, this is also equal to 
$[\psi(a)]_{\tau(s)}$, and similarly for $b$ and $c$. Thus $[\psi(a)]_{\tau(s)}$, $[\psi(b)]_{\tau(s)}$, 
and $[\psi(c)]_{\tau(s)}$ are distinct, and $B_{\tau(s)}([\psi(a)]_{\tau(s)}; [\psi(b)]_{\tau(s)}, [\psi(c)]_{\tau(s)})$, 
so by definition of the $L$-relation on $B(r_2)$, we have $L_2(\psi(a); \psi(b), \psi(c))$. 

Conversely, suppose that $L_2(\psi(a); \psi(b), \psi(c))$. Since $\psi$ is injective, $a$, $b$, 
and $c$ are distinct. By the final part of Lemma \ref{3.3}, $L_1$ holds between $a, b,$ and $c$ 
in some order. The argument already given shows that $L_2$ holds between $\psi(a)$, $\psi(b)$, and 
$\psi(c)$ in that same order. But as $L_2(\psi(a); \psi(b), \psi(c))$, it must be the same 
for $a$, $b$, $c$, i.e. $L_1(a; b, c)$ as required (or else $L_1(a; c, b)$, which is equivalent). \end{proof}

\section{Amalgamation arguments and the existence of the generic}

Now we are able to verify the amalgamation property for finite $C$-coloured trees of $B$-sets, and 
deduce that there is a `generic' such, in the style of Fra\"iss\'e's Theorem. The desired Jordan
group will then be the group of automorphisms of this limit object. This is not however quite good 
enough, since we need to include in the limit the set on which the Jordan group acts. This will be
provided by Lemma \ref{3.8}, and indeed the relevant set will carry an $L$-relation on it, which is 
preserved by the group action. Note that in \cite{Bhattacharjee} and \cite{Bradley}, although the 
close connection between finite trees of $B$-sets and $L$-structures was explained, the Fra\"iss\'e 
part of the proof was done entirely in the terms of the $L$-structures. This had the advantage that 
as these are first order structures, a direct appeal to a version of Fra\"iss\'e's Theorem in the 
literature immediately supplied the existence of the limit. The down-side was that complicated 
machinery was required to recover the desired tree of $B$-sets from the limit as $L$-relation. Here 
we seek to minimize the difficulties by working directly with the trees of $B$-sets. Categorical generalizations of Fra\"iss\'e's Theorem were given by M. Droste and R. G\"obel in \cite{Droste2} and \cite{Droste3}. Since then, the notion that certain categories enjoying appropriate forms of the amalgamation property (among other conditions), are natural domains for idealizations of Fra\"iss\'e theory has been resonating through the literature, see \cite{Kubis}, \cite{Caramello}, \cite{Kubis2}. In all these categorical generalizations, the classical Fra\"iss\'e construction is recovered upon considering amalgamation in a category whose finite objects are (finite) structures in a given first order language, the (monic) morphisms are (model theoretic) embeddings between them, and co-limits of sequences are unions of chains of structures. As trees of $B$-sets are not themselves first order structures, amalgamating them directly requires moving outside the classical Fra\"iss\'e setting, though the limit we construct has a countable set as domain, on which its automorphism group acts.

%%% and \cite{Kubis}) that Fra\"iss\'e limits can be successfully carried out (as in those papers) in very many categories outside the classical model theoretic context; the classical context being categories of finite $\mathcal{L}$-structures with $\mathcal{L}$-embeddings with the amalgamation property. We establish how to consider classes of finite trees of $B$-sets in a way that allows for amalgamation, and hence a Fra\"iss\'e limit.

With this in mind, we let $\mathcal C$ be the family of finite trees of $B$-sets over our ambient 
$C$-coloured structure tree $T_C$, which are the objects (in the categorical language). As embeddings (or morphisms) between members of $\mathcal C$ we take strong 
embeddings.  Since these are complicated structures, we reduce the verification of amalgamation to 
that of 1-point extensions. If $A, E \in {\mathcal C}$ and $A$ is a strong substructure of $E$, 
with the set at the root of $E$ having just one more point  than for $A$, then $E$ is said to 
be a {\em 1-point extension} of $A$. We now list the possible 1-point extensions:

In a {\em star} extension, the root $r$ of $T^E$ is strictly below the root $s$ of $T^A$, and all 
points of the $B$-set $(B(r),B_r)$ in $E$ except for $e$ are leaves, which are all attached to 
$e = f_r(s)$ (which is its single ramification point), and the function $g_{rs}$ is a bijection 
from $B(r) \setminus \{e\}$ to $B(s)$. 

The other types of 1-point-extensions are all `root extensions', meaning that the roots of $T^A$ and 
$T^E$ are equal, to $r$ say. Since we are insisting that all nodes are populated, $B(r)$  must 
in $E$ have at least 2 vertices, and by our convention on $B$-relations and graphs, must have at least 
one edge. There are four cases, depending on the valency of $e$ and of its neighbour or neighbours.

In a {\em leaf} extension, a new vertex $e$ is added as a leaf of $B_r$ joined to an existing node 
$u$ which is a leaf of $B_r$  in $A$. There is no change in the tree, since no ramification points 
are added, and any branch containing $u$ now also has $e$ included (so this is a `trivial' 
modification, involving just altering the `labels'). 

In a {\em dyadic} extension, an edge with endpoints $u$, $v$ in $B(r)$ is replaced by a pair of edges 
$ue$ and $ev$ where $e$ is a new vertex (and the edge $uv$ is deleted). Again the tree is not changed,
but just the `labelling'.

In a {\em ternary} extension, $e$ is attached to a vertex $u$ of $B(r)$ which in $A$ is dyadic, so that 
$u$ becomes a ramification point. The tree therefore has one node $f_r^{-1}(u)$ added, at which there is
a 3-point linear $B$-relation.

In a {\em ramification} extension, $e$ is attached to a point $u$ of $B(r)$, which in $A$ is a ramification 
point. There are no new ramification points at the root, but it is possible for ramification points to be 
added higher up.

\begin{lemma} \label{4.1} Any 1-point extension $E$ of $A$ for non-empty members $A$ and $E$ of $\mathcal C$ 
is a star extension, or a root extension of one of the stated types. \end{lemma}

\begin{proof} Let $r$ and $s$ be the roots of $E$ and $A$ respectively. Since $E$ is an extension of $A$,
$T^A \subseteq T^E$, so $r \le s$. In the first case, $r < s$, in which case we let $e = f_r(s)$ in $E$. 
Then $g_{rs}$ maps $B^E(r) \setminus \{e\}$ surjectively to $B^E(s)$. Thus 
$|B^A(s)| \le |B^E(s)| \le |B^E(r)| - 1$. Since $E$ is a 1-point extension of $A$, $g_{rs}$ is a bijection 
from $B^E(r) \setminus \{e\}$ to $B^E(s) = B^A(s)$, and it follows that $E$ is a star extension of $A$.

Otherwise, $r = s$, and we let $B^E(r) = B^A(r) \cup \{e\}$. As in \cite{Bhattacharjee} we show that $e$ 
has valency 1 or 2 in the $B$-set $(B^E(r), B_r^E)$, though we replace the appeal to a quaternary relation 
to considerations of strong substructures. Note that $e$ has at least one neighbour, as otherwise it would
be isolated, and $A$ empty. Suppose for a contradiction that $e$ has at least 3 neighbours $x, y, z$ say in 
$B_r^E$. Then $e$ is the centroid of $\{x, y, z\}$ in $B^E_r$, so as $B^A_r$ is a strong substructure of
$B^E_r$, $e$ must also lie in $B^A(r)$, contrary to its choice. 

If the valency of $e$ is 1 then the extension is a root extension which is a leaf, ternary, or 
ramification extension of $A$ depending whether the valency of the point it is attached to in 
$B_r^A$ is 1, 2, or greater than 2. If its valency is 2, then $E$ is a dyadic extension. (See 
the proof of Lemma \ref{4.3} for more details.)         \end{proof}

%\begin{lemma} \label{4.2} The amalgamation property holds for one-point extensions: given one-point extensions $E_1$ and $E_2$ of $A$ in $\mathcal C$, there is some $E \in \mathcal C$ strongly embedding $E_1$ and $E_2$ commuting with the embeddings of $A$ in $E_1$ and $E_2$. \end{lemma}

\begin{lemma} \label{4.2} 1-point extensions can be amalgamated in $\mathcal C$. \end{lemma}

\begin{proof} Note that throughout this proof, `extension' means `strong arboreal extension'. Let $E_1$ and $E_2$ 
be 1-point extensions of $A$, and let the extra points be $e_1$ and $e_2$ respectively. We have to show how to
amalgamate $E_1$ and $E_2$ over $A$. Note that when we say that various nodes are handled `trivially', we mean that 
what happens there and above in an extension is `essentially' the same as in the original structure, and the labels 
only differ by addition of new vertices in the same class as an existing one, so there is no change up to isomorphism.
One checks that even in cases where $T$ has maximal elements, the amalgamation can be performed (the finite trees may
get `fatter', but not `taller'). 

%We remark that, unlike in \cite{Bhattacharjee}, there is a situation we deal with in the following proof, involving two ternary extensions at the same vertex, in which the amalgam $E$ we provide is not a 1-point extension of either $E_1$ nor $E_2$; however it will strongly embed both $E_1$ and $E_2$ and serve as their amalgam over $A$.

If the map fixing $A$ pointwise and taking $e_1$ to $e_2$ is an isomorphism, then we can identify $e_1$ 
and $e_2$ to obtain an amalgam trivially, so from now on we assume that this is not an isomorphism.

We need to consider many cases for the possible types of extension for $E_1$ and $E_2$, though most of 
them are treated in a similar way. We first note that they cannot both be star extensions having roots with 
the same colour, since these would be isomorphic, contrary to assumption. If they are star extensions having
roots $r_1$, $r_2$ with different colours $c_1$ and $c_2$, with for instance $c_1 < c_2$, then we form an
extension $E$ of $E_2$ having an extra node $r_1$ below $r_2$, with a star at $r_1$ as for $E_1$ but with 
$e_2$ added as an extra leaf.  

Next suppose that $E_1$ is a star extension and $E_2$ is a root extension.  Let $r$ be the root of $T^{E_1}$. 
Then $s = f_r^{-1}(e_1)$ is the root of $T^A$, and hence also the root of $T^{E_2}$. An amalgam $E$ of $E_1$ and 
$E_2$ over $A$ having root $r$ is obtained from $E_2$ by adding a star node below it at $r$, where the $B$-set 
is as for $E_1$, with the addition of one extra leaf $e_2$ attached to $e_1$. One checks that the obvious 
embeddings of $E_1$ and $E_2$ into $E$ are strong.

From now on we suppose that $E_1$ and $E_2$ are non-isomorphic root extensions. Let $r$ be the common 
root of $T^A$, $T^{E_1}$, and $T^{E_2}$, and let $u, v$ be the unique vertices of $B(r)$ in $A$ that $e_1, e_2$ 
are joined to, respectively, or in the dyadic case, $u_1, u_2$, or $v_1, v_2$ as the case may be. The most 
straightforward case is that in which all these points are distinct. Then we can amalgamate by taking the
union of $E_1$ and $E_2$. More precisely, at the root, $e_1$ and $e_2$ are added correctly to $B(r)$, and $T^E$
is taken to be the union of $T^{E_1}$ and $T^{E_2}$, which overlap just in $T^A$, so there is no conflict.

In the remaining cases, $E_1$ and $E_2$ are non-isomorphic, but $u = v$ (or $u_1 = v_1, u_2 = v_2$ in the dyadic 
case). We note that in all except the case of a ternary or ramification extension, $E_1$ and $E_2$ are now actually 
isomorphic over $A$. For a leaf extension, $e_1$ is a leaf of $B^{E_1}(r)$ and $u$ is dyadic in $B^{E_1}(r)$, and 
so no ramification points are added, so $T^{E_1} = T^A$, and similarly for $E_2$. If the extension is dyadic, so 
that $u_1 u_2$ is a edge in $B^A(r)$, in $E_1$ $u_1 e_1$ and $e_1 u_2$ are edges, so that $e_1$ is dyadic, and the 
ramification orders of $u_1$ and $u_2$ are unchanged. Hence once again the trees $T^A$ and $T^{E_1}$ are equal,
and if $u_1$ is a ramification point, the branch at $u_1$ containing $u_2$ is replaced by the branch containing
$e_1$, and similarly for $u_2$. Exactly the same applies in $E_2$, and hence $E_1$ and $E_2$ are isomorphic over
$A$ by the map taking $e_1$ to $e_2$. 

For the ternary case, $u \in B^A(r)$ is dyadic, with neighbours $x$ and $y$ say, and $e_1$ is a leaf adjacent
to $u$ in $B^{E_1}_r$ and $e_2$ is a leaf adjacent to $u$ in $B^{E_2}_r$. Thus $u$ is a ramification point
in both $B^{E_1}_r$ and $B^{E_2}_r$, and hence $s_1 = f_r^{-1}(u)$ evaluated in $E_1$ and $s_2 = f_r^{-1}(u)$ 
evaluated in $E_2$ both exist, and each of $B^{E_1}_{s_1}$ and $B^{E_2}_{s_2}$ is a 3-element $B$-set, hence a
linear $B$-set. If $s_1$ and $s_2$ have the same colour $c$, then we can most easily amalgamate by allowing  $s= s_1 = s_2$ in $E$ and finding a linear 
$B$-relation at $s$ which extends those on $B^{E_1}_{s_1}$ and $B^{E_1}_{s_2}$. Otherwise, assume 
without loss of generality that $c_1 < c_2$ where $c_1$ and $c_2$ are the colours of $s_1$ and $s_2$. We form 
an amalgam $E$ such that $T^E = T^{E_1} \cup T^{E_2} = T^A \cup \{s_1, s_2\}$. To make things work, we require 
one extra `auxiliary' point $e$ which we add to the root $B$-set. So we let $B^E(r) = B^A(r) \cup \{e_1, e_2, e\}$.
We let $B^E_{s_2} = B^{E_2}_{s_2}$ be unchanged (strictly speaking, it is obtained from $B^{E_2}_{s_2}$
by a `trivial' modification). In $B^E_r$, $e_1$, $e_2$, and $e$ are all leaves attached to $u$, so that
now $u$ has ramification order 5, and there are 5 branches at $u$ which give rise to 5 members of $B^E(s_1)$.
Now the relation between $[e_1]_{s_1}$, $[x]_{s_1}$, and $[y]_{s_1}$ in $B^E_{s_1}$ is already known, since 
this has to be an extension of $B^{E_1}_{s_1}$, which must be a linear $B$-relation on 3 points. We decree that 
$B^E_{s_1}$ will have just one ramification point $[e]_{s_1}$, with valency 3. The key point is that 
$[x]_{s_1}$, $[y]_{s_1}$, and $[e_2]_{s_1}$ lie in distinct branches. Which branch $[e_1]_{s_1}$ lies in depends 
on how it is related to $x$ and $y$. If $B^{E_1}_{s_1}([x]_{s_1}; [e_1]_{s_1}, [y]_{s_1})$ or 
$B^{E_1}_{s_1}([e_1]_{s_1}; [x]_{s_1}, [y]_{s_1})$, then $[e_1]_{s_1}$ is in the same branch as $[x]_{s_1}$, 
and if $B^{E_1}_{s_1}([y]_{s_1}; [e_1]_{s_1}, [x]_{s_1})$ it is in the same branch as $[y]_{s_1}$. We now 
let $f^{-1}_{s_1}([e]_{s_1}) = s_2$ in $E$, and then $[x]_{s_1}$, $[y]_{s_1}$, and $[e_2]_{s_1}$ lie in 
distinct branches of $[e]_{s_1}$ in $B^E_{s_1}$, and hence $[x]_{s_2}$, $[y]_{s_2}$, and $[e_2]_{s_2}$
are distinct members of $B^E(s_2)$ (as is required). It can be easily checked that the embeddings of $E_1$ 
and $E_2$ into $E$ thus defined are strong.

Finally we deal with the case in which $E_1$ and $E_2$ are non-isomorphic ramification extensions, and $e_1$ and 
$e_2$ are both attached to $u$, which has ramification order at least 3 in $B^A_r$, where $r$ is the root of $T^A$
(hence also of $T^{E_1}$ and $T^{E_2}$). Let $s, s_1, s_2$ be $f_r^{-1}(u)$ evaluated in $A$, $E_1$, and $E_2$
respectively. Since $T^A \subseteq T^{E_1}, T^{E_2}$ we have $s \in T^{E_1}, T^{E_2}$. 

Let $x_1, x_2, \ldots, x_N$ be the neighbours of $u$ in $B^A_r$. Since $A$ is a substructure of 
$E_1$,$[x_1]_s, \ldots, [x_N]_s$ are distinct members of $B^{E_1}(s)$, from which it follows that $s_1 \le s$. Also
$|B^A(s)| = N$ and $|B^{E_1}(s_1)| = N + 1$, so that either $s_1 = s$, or $s$ is a successor of $s_1$ in $T^{E_1}$;
furthermore, in this latter case, $B^{E_1}_{s_1}$ is a star with centre $[e_1]_{s_1} = f_{s_1}(s)$. Similarly for $E_2$. We now consider cases in turn, depending on the relationship between $s_1, s_2$, and $s$.

\noindent{\bf Case 1}: $s_1 = s_2 = s$. The trees of $B$-sets obtained from $E_1$ and $E_2$ over $A$ obtained by
removing the root $r$ are 1-point (root) extensions, and $|B^A(s) <  |B^A(r)|$, so we may amalgamate these 
inductively. The root $B$-set is then added to the amalgam, with both $e_1$ and $e_2$ leaves attached to $u$.

\noindent{\bf Case 2}: $s_1 < s$ and $s_2 < s$. Then since $E_1$ and $E_2$ have stars at $s_1$, $s_2$ respectively, 
we see that since $E_1$ and $E_2$ are assumed non-isomorphic over $A$, the colours $c_1$ of $s_1$ and $c_2$ of
$s_2$ are unequal. Suppose $c_1 < c_2$. We take $T^E = T^A \cup \{s_1, s_2\}$. At $r$, $E$ has $e_1$ and $e_2$
both attached as leaves to $u$. At $s_1$, starting with $B^{E_1}_{s_1}$, we adjoin $[e_2]_{s_1}$ as an extra leaf  
attached to $[e_1]_{s_1}$, and $B^E_{s_2} = B^{E_2}_{s_2}$.

\noindent{\bf Case 3}: $s_1 < s = s_2$ (and similarly $s_2 < s = s_1$). In this case, if we remove the root $r$, then 
$E_1$ is a star extension of $A$ since $B^{E_1}_{s_1}$ is a star centred at $[e_1]_{s_1}$, and $E_2$ is a root
extension of $A$. Since $|B^A(s)| <  |B^A(r)|$, we may again appeal to the induction hypothesis to amalgamate, and then
restore the root $r$ again obtained from $B^A_r$ adding $e_1$ and $e_2$ as leaves attached to $u$. For clarity 
we make this more explicit. We may take the tree $T^E$ to be $T^{E_2} \cup \{s_1\}$. At $r$, $E$ has the $B$-set
obtained from $B^A_r$ by adjoining $e_1$ and $e_2$ as distinct leaves attached to $u$. At $s_1$ $E$ has the
$B$-set obtained from $B^{E_1}_{s_1}$ by adjoining $[e_2]_{s_1}$ as an extra leaf attached to $[e_1]_{s_1}$.
Otherwise $E$ is the same as $E_2$.    \end{proof}

\vspace{.1in}

To show that the family of finite trees of $B$-sets over a given ambient tree $T$ has the amalgamation 
property, it remains to show that any extension can be realized by a sequence of 1-point extensions, 
and that consequently, any amalgamation can be performed by composing amalgamations of the 1-point 
extensions. These steps are carried out again using \cite{Bhattacharjee} as a guide, modified 
using our current approach. 

\begin{lemma} \label{4.3} If $A$ and $E$ are finite trees of $B$-sets, such that $A$ is a strong substructure 
of $E$, then there is a finite sequence $A= A_0, A_1, \ldots, A_{n-1}, A_n = E$ such that each $A_i$ is a 
strong substructure of $A_{i+1}$ ,and $A_{i+1}$ is a 1-point extension of $A_i$.  \end{lemma}

\begin{proof} We use induction on the number of nodes of $T^E$. If this is 1, then the single node carries a
linear betweenness relation, and the result is immediate. 

Assume that $A < E$. We find a (strong) 1-point extension $A'$ of $A$ which is a strong substructure of $E$, 
and then we can repeat. Let $r$ and $s$ be the roots of $T^A$ and $T^E$ respectively. Then $s \le r$.

In the first case, $s < r$. Let $e = f_s(r)$. For each $a \in B^A(r)$, let $\psi(a) \in g_{sr}^{-1}(a)$ be at distance
1 from $e$. Thus $\psi(a) \in B^E(s)$ and $[\psi(a)]_r = a$. Let $t \in T^E$ be the successor of $s$ which is
below $r$. Since all $[\psi(a)]_r$ are distinct, so are $[\psi(a)]_t$, so the $\psi(a)$ must lie in distinct branches
at $e$ in $B^E_s$. Let $A'$ be the star extension of $A$ contained in $E$ given as follows. We obtain $T^{A'}$
from $T^A$ by adjoining $s$ as root, and we let $B^{A'}(s) = \{e\} \cup \{\psi(a): a \in B^A(r)\}$ and take for 
$B^{A'}_s$ the star centred at $e$. One checks that $A'$ is a strong 1-point star extension of $A$ which is a strong 
substructure of $E$.

Now suppose that $s = r$. Then $B_r^A$ is a strong substructure of the positive $B$-set $B^E_r$, and since $A < E$, 
this is a proper substructure. There must be at least one $x \in B^E(r) \setminus B^A(r)$ at distance 1 to to a 
member of $B^A(r)$, and we need to consider the set $X$ of all such. If $x \in X$, let $Y_x$ be the set of members 
of $B^A(r)$ at distance 1 from $x$. Then $|Y_x| \le 2$, since if there are three members of $B^A(r)$ at distance 1 
from $x$, their centroid would have to be $x$, contrary to $A$ a strong substructure of $E$. We write $Y_x$ as
$\{u_x\}$ or $\{u_x, v_x\}$ depending on whether $|Y_x| = 1$ or 2. If some $|Y_x| = 1$ and $u_x$ is a leaf of 
$B^A_r$, then we can let $T^{A'} = T^A$, and replace the $B$-set at the root by the restriction of $B^E_r$ to 
$B^A(r) \cup \{x\}$. Since there are no new ramification points, this gives a strong 1-point extension of $A$ 
which is a strong substructure of $E$. The only change is `trivial', in that any branch containing $u_x$ has $x$ 
added to it. A similar argument applies if some $|Y_x| = 2$ (the dyadic case), where again $T^{A'} = T^A$, and the 
only change is `trivial' in the labelling, as if for instance $u_x$ is a ramification point, then the branch 
containing $v_x$ now has $x$ added to it, and similarly for $v_x$ a ramification point. If $|Y_x| = 1$ and $u_x$
is dyadic in $B^A(r)$, then $u_x$ becomes a ramification point in $A'$. So we add a new node $t$ to the tree $T^A$, 
where $t = f_r^{-1}(u_x)$. Since $u_x$ now has valency 3, a 3-element linear $B$-set is induced on $B^{A'}(t)$ from 
$B^E_t$.

Now we suppose that for every $x \in X$, $|Y_x| = 1$ and $u_x$ is a ramification point of $B^A_r$. We don't need to 
alter this part of the tree, but $u_x$ has acquired an extra branch, and so $B^{A'}(t)$ has one more vertex $y_x$ 
than $B^A(t)$, where $t = t_x = f_r^{-1}(u_x)$, which is the branch at $u_x$ containing $x$. Then 
$y_x \in B^E(t) \setminus B^A(t)$, and every member of $B^E(t) \setminus B^A(t)$ arises as $y_x$ from some 
$x \in X$ such that $t_x = t$. We now apply the induction hypothesis to the subtree of $T^E$ having root $t$, which 
has fewer nodes that $T^E$. We deduce that there is a strong 1-point extension of the tree of $B$-sets on $A$ 
having root at $t$ which is a strong subtree of that on $E$, obtained by adjoining one of the $y_x$s for which
$t_x = t$. We adjoin this $x$ to $B^A(r)$ to form the $B$-set $B^{A'}(r)$ at the root induced from $B^E_r$, 
completing the definition of $A'$, which is a strong 1-point extension of $A$, and is a strong substructure of $E$.
(The complication here arises because we don't know in advance which $y_x$ will be provided by the induction 
hypothesis, so we have to wait for this before making our final choice of $x$.)
          
We can now repeat this finitely many times to give the desired sequence.   \end{proof}

\begin{corollary} \label{4.4} The family of finite trees of $B$-sets over the ambient tree $T$ has the 
amalgamation property (with respect to strong embeddings). \end{corollary}

\begin{proof} As in \cite{Bhattacharjee}, this follows from Lemmas \ref{4.2} and \ref{4.3}. 
There is a slight additional complication, in that unlike in \cite{Bhattacharjee}, in the case of amalgamating two 1-point extensions $E_1$ and $E_2$ over a common dyadic vertex of $A$ specifiying different colours $c_1$ and $c_2$, we had to adjoin an auxiliary point $e$ to make the argument go through. This means that, in such cases, $E$ is a 2-point extension of each of $E_1$ and $E_2$, hence a composition of two
1-point extensions. However this problem does not persist, since a dyadic vertex always becomes a ramification point after a ternary extension at that vertex, so any finite sequence of 1-point extensions (as appearing in Lemma \ref{4.3}) involves at most one ternary extension at any given vertex. Hence we can nevertheless carry out the argument as in \cite{Bhattacharjee} at the cost of some extra steps.
\end{proof}

\begin{lemma} \label{4.5} The family of finite trees of $B$-sets over the ambient tree $T_C$ has the joint 
embedding property.  \end{lemma}

\begin{proof} This is done just as in  \cite{Bhattacharjee}. Let $A_1$ and $A_2$ be two finite trees of $B$-sets, 
with roots $r_1$ and $r_2$ and $B$-sets there of sizes $m$ and $n$ respectively. By increasing $A_1$ and $A_2$ if 
necessary we may assume that $m, n \ge 3$. Choose $c \in C$ which is less than the colours of $r_1$ and $r_2$. 
We form an extension $E$ which has root $r$ coloured $c$, and above this, disjoint copies of $A_1$ and $A_2$. At 
$r$, $E$ has a $B$-set which has two ramification points $e_1$ and $e_2$ which are joined by an edge, and the 
other vertices are $\{x_i: 1 \le i \le m-1\} \cup \{y_j: 1 \le j \le n-1\}$, all leaves, with $x_i$ joined to 
$e_1$ and $y_j$ joined to $e_2$. The two nodes of $T^E$ above $r$ are $r_1$ and $r_2$, and $f_r(r_1) = e_1$, 
$f_r(r_2) = e_2$, $g_{r r_1}$ maps  $\{x_i: 1 \le i \le m-1\} \cup \{e_2\}$ bijectively to the points of  
$B(r_1)$ and  $g_{r r_2}$ maps $\{y_j: 1 \le i \le n-1\} \cup \{e_1\}$ bijectively to the points of $B(r_2)$.          
  \end{proof}

We now wish to deduce the existence of a Fra\"iss\'e limit of the family $\mathcal C$ of finite trees 
of $B$-sets over the ambient tree $T_C$ derived from the chain $C$ under the relation of strong 
embedding. The method given in \cite{Bhattacharjee} was to quote a version of Fra\"iss\'e's Theorem 
as presented in \cite{Evans}. To appeal to this directly, the class has to be over a first order 
language. Since we are avoiding use of $L$-structures, we have elected to reprove the relevant limit result 
directly, for which we can go through the same steps as in \cite{Evans} adapted for the present 
context. Such steps are part of the common core of methods given in considerably more general categorical versions of Fra\"iss\'e's Theorem in e.g. \cite{Droste2}, \cite{Droste3}, \cite{Kubis}, \cite{Caramello}, or \cite{Kubis2}, allowing the construction of chain $A_i$, such that the union $A$ of this chain satisfies the numbered properties below. We further verify that the union of this sequence is itself a countable tree of $B$-sets.

\begin{theorem} \label{4.6} Let $T_C$ be the countable coloured tree with no least element obtained 
from a countable chain $C$ with no least as given above, in which all nodes ramify infinitely, and let 
$\mathcal C$ be the family of finite trees of $B$-sets over $T_C$, with embeddings being strong 
embeddings. Then there is a unique countable tree of $B$-sets $A$ such that

(i) every finite tree of $B$-sets which is a strong substructure of $A$ lies in $\mathcal C$,

(ii) $A$ is a union of a chain of members of $\mathcal C$,

(iii) every member of $\mathcal C$ embeds strongly into $A$,

(iv) if $B$ and $E$ are members of $\mathcal C$ such that $B$ is a strong substructure of $E$, then 
any strong embedding of $B$ into $A$ extends to a strong embedding of $E$ into $A$,

(v) every isomorphism between finite strong substructures of $A$ extends 
to an automorphism of $A$.     \end{theorem}

\begin{proof} We note that there are only countably many members of $\mathcal C$ up to (arboreal) 
isomorphism. In addition for any $B$ and $E$ in $\mathcal C$, there are only finitely many strong 
embeddings of $B$ into $E$ (the argument would still work, provided that there were only countably 
many). We form an increasing sequence $A_0 \le A_1 \le A_2 \le \ldots$ of members of $\mathcal C$, 
such that each is a strong substructure of the next, and take $A$ for the union of the chain. Thus 
(ii) is fulfilled. We let $A_0$ be any member of $\mathcal C$. Enumerate all of $\mathcal C$ up to 
isomorphism in a countable sequence, and at odd stages, ensure using Lemma \ref{4.5} that each 
particular member of $\mathcal C$ embeds strongly in some $A_i$. This guarantees the truth of (iii). 
Furthermore, any finite tree of $B$-sets which is a strong substructure of $A$, must be contained in
some $A_i$, of which it must be a strong substructure, and hence also lies in $\mathcal C$. Thus (i) holds.

At even stages $i$ from 2 on, we ensure that all strong embeddings extend. List all strong embeddings 
between members of $\mathcal C$  so that each is listed infinitely many times. At a given stage, suppose 
that $\theta: B \to E$ is a strong embedding between members of $\mathcal C$. If $B$ does not embed 
strongly in $A_i$, we let $A_{i+1} = A_i$. Since $B$ embeds as a strong substructure of some $A_j$, and 
since the embedding is listed infinitely many times, the embedding will appear in the list at some even 
stage $k$ beyond the $j$th, {\em then} we can extend non-trivially. Precisely, we amalgamate the strong 
embeddings of $B$ into $E$ and into $A_k$ (both inclusion), and this results in a member $A_{k+1}$ of 
$\mathcal C$ and strong embeddings of $E$ and $A_k$ into $A_{k+1}$ `making the diagram commute'. By 
replacing by an isomorphic copy, we may assume that the embedding of $A_k$ into $A_{k+1}$ is inclusion. 

This concludes the inductive construction of all the $A_i$s, and hence of their union $A$. As it a countable union of finite trees of $B$-sets, $A$ is itself countable.

We proceed to show that $A$ is a tree of $B$-sets. Let $T_i = T^{A_i}$ be the finite $C$-coloured 
tree populated by  $A_i$. Then $(T_i)_{i \in \omega}$ is an increasing sequence of finite 
$C$-coloured trees, and it follows easily that $T = \bigcup_{i \in \omega}T_i$ is a $C$-coloured 
tree, since all the properties required for this are verified by looking at finitely many points 
at a time. We show that $T \cong T_C$. 

First, all points ramify infinitely, except on the top level (if it exists). To see this, fix a
non-maximal $t \in T$ and finite $n \ge 3$. Consider a finite tree containing $t$ and $t_1, \ldots, t_n$ 
such that for each $i \neq j$, $t_i \curlywedge t_j = t$. We find $B \in {\mathcal C}$, with a $B$-set 
at the root $t$, having exactly $n$ ramification points $a_1, \ldots, a_n$ and $f_t(t_i) = a_i$. By 
(iii) this embeds strongly into $A$, and so (using also (v)) $T$ has ramification order at least $n$ at 
$t$. Since $n$ was arbitrarily large, $T$ ramifies infinitely at $t$. Points on the top level do not 
ramify at all.

Next, any point $t$ of $T$ lies in a maximal chain which is isomorphic to $C$. For let 
$C = \{c_n: n \in \omega\}$. Then by using amalgamation, it can be seen that there are 
$m_0 < m_1 < m_2 < \ldots$ such that $A_{m_n}$ contains a chain containing $t$ in which all colours 
in $\{c_i: i \le n\}$ occur, and it can be arranged that these chains are nested, and hence their 
union is a chain isomorphic to $C$. 

Next we have to define the $f$ and $g$ maps. For any $s \in T$, and $t > s$,  there is $i$ such 
that $s, t \in T_i$, and hence in $T_i$, $f_s(t)$ exists and lies in $B(s)$. Since $A_i$ is a strong 
substructure of each $A_j$ for $j \ge i$, the value of $f_s(t)$ stays the same in $A_j$, and hence 
we may unambiguously define it to be the same in $A$, and the property that $f_s(t_1) = f_s(t_2)$ if 
and only if $t_1$ and $t_2$ lie in the same cone at $s$ is inherited from its truth in $A_j$ for a large 
enough $j$. If $s = t$, $g_{s s}$ is the identity on $B(s)$. If $s < t$, to define $g_{s t}(x)$ where 
$x \in B(s) \setminus \bigcup_{s \le u < t}g^{-1}_{s u} (f_u(t))$, we find $i$ large enough so that 
$s, t \in T_i$, and $x \in B(s)$ in $A_i$, and let  $g_{s t}(x)$ take its value as in $A_i$. Since
$A_i$ is a strong substructure of subsequent $A_j$s, this is well-defined, and the required properties 
are verified.

It remains to verify properties (iv) and (v).  

To prove (iv), note that by (iii), $B$ embeds strongly into $A$, so by replacing $B$ by a copy, we may assume 
that $B$ is a strong substructure of $A$. By construction, the strong embedding of $B$ into 
$E$ extends to a strong embedding into some $A_i$, and hence into $A$.

Property (v) may be proved by back-and-forth using the extension property which has been arranged during 
the construction.    \end{proof}

\section{Analyzing the limit structure}

Now we study the limit structure which Theorem \ref{4.6} showed exists. Unlike in \cite{Bhattacharjee} we 
do not expect this to be $\aleph_0$-categorical, in view of the presence of the infinitely many colours $C$ 
as labels of the nodes of $T_C$. We defined `tree of $B$-sets' so that it would apply also to the infinite 
case, specifically to the limiting structure as well as all the finite approximations. The fact that $A$ 
takes this form was verified in Theorem \ref{4.6}, its automorphisms are the strong isomorphisms $A$ to $A$. As explained in the discussion, it is important that 
there is an associated $L$-structure $(M, L)$, which the automorphism group of $A$ acts upon. Another key 
feature is that we can extend the $[a]_t$ notation, so that now this will be an infinite subset of the 
domain $M$ of the $L$-relation, and the members of $B(t)$ will be all the possible $[a]_t$, for fixed $t$, which will be pairwise disjoint subsets of $M$. We shall also remark that the $B$-set on each $B(t)$ is infinite 
and dense in which each vertex has ramification order $\aleph_0$. This last is immediate from the observation 
that for any $c \in C$, and $n$, there is a finite tree of $B$-sets in which the $B$-set at some node coloured 
$c$ contains a `large' finite tree, namely one in which all points are at distance up to $n$ from a fixed 
`centre', and has valency $n$ at all non-leaf vertices. One now amalgamates to show that this behaviour must 
apply at every node, and hence in $A$, the $B$-set at every node has all vertices with ramification order 
$\aleph_0$. In a similar way one establishes density. An important semilinear order, $K$, introduced earlier, 
arises as each branch at each node $[a]_t$ of $(B(t), B_t)$, which has maximal chains isomorphic to 
$\mathbb Q$, and ramifies at each point with infinite ramification order. 

To ease notation, we write the ($C$-coloured) tree populated by $A_i$ as $T_i$, and $L_i$ the $L$-relation on 
$B(r_i)$ in $A_i$, where $r_i$ is the root of $T_i$. Now to obtain the domain $M$ of an $L$-relation from those 
of the $A_i$, we appeal to Lemma \ref{3.8}, which tells us that the $L$-relation $L_i$ can be embedded in 
$L_{i+1}$, so we can therefore form the union of all their domains (more formally, their direct limit), and 
obtain a natural $L$-relation on this union, given by taking the union of all the individual $L_i$s.

Next to make sense of the notation $[a]_t$ for $t \in T$, where $a \in M$, we find $i$ so that $a$ lies in the 
domain of $A_i$, and $t$ in $T_i$, where $[a]_t$ already has a meaning (provided that $a$ lies in the domain of 
$g_{rt}$ where $r$ is the root of $T_i$). This is dependent upon $i$, but as $i$ increases, the fact that we 
have a sequence of strong substructures tells us that the values of $[a]_t$ also increase, and so in $A$ we 
may just take the union of all these finite sets. Note that as in each $A_i$ any two $[a]_t$ for the same value 
of $t$ are either disjoint or identical (as explained in the discussion before Lemma \ref{3.3}), the same applies 
in $A$.

We remark that in \cite{Bhattacharjee} the `official' structure constructed was the $L$-structure $(M, L)$, and 
some detailed calculations demonstrated how the tree of $B$-sets $A$ could be represented from the $L$-structure. 
We have chosen to do things the other way round, or rather to handle the $L$-structure simultaneously. In section 6 
we shall show that one can `recover' $A$ from $(M, L)$, as in \cite{Bhattacharjee}.
%, and it then follows that the automorphism group of $(M, L)$ is the same as (isomorphic to) the automorphism group of $A$. [WARNING: I am not sure how $(M,L)$ can, in general, recover the colours of $A$, yet I am sure we are using the notion automorphism of $A$ to be colour preserving (i.e. bijective strong embeddings of $A$). Particularly in light of Lemma \ref{3.7}, it seems that if the tree admits order isomorphisms that don't preserve colours, Aut($M,L$) might contain permutations of $M$ that don't correspond to (colour preserving) automorphisms of $A$. So while Aut(A) corresponds to a subgroup of Aut(M,L), are we sure that it is (always) the same?]   
All that we require 
at this stage is that $G = {\rm Aut}(A)$ {\em acts} (faithfully) on the set $M$, which it does by Lemma \ref{3.7}, and preserves $L$, which it does by Lemma \ref{3.8}. This enables 
us to work in practice with $A$ rather than $(M, L)$. All our arguments in this section will be about the tree 
of $B$-sets $A$, rather than the associated $L$-relation.

For each $t \in T$, we let $S_t = \bigcup\{[a]_t: a \in M\}$. This is called the {\em pre-$B$-set} at $t$. 
%In \cite{Bhattacharjee} the semilinear order on $T$ is {\em defined} by saying that $s \le t$ if and only if $S_s \supseteq S_t$. We already have the semilinear order from our construction, so we do not need this equivalence.

Saying that $S_t$ is a `pre-$B$-set' is meant to suggest that after we have taken the quotient by a suitable 
equivalence relation on it, it becomes (the domain of) a $B$-set. The equivalence relation $E_t$ in question, 
following the notation of \cite{Bhattacharjee}, is the one whose classes are the $[a]_t$. Thus 
$S_t/E_t = \{[a]_t: a \in S_t\} = B(t)$, and as above we have the $B$-relation $B_t$ on $B(t)$. Here $G_t$ is the subgroup of automorphisms of $A$ fixing $t \in T$.

A subset $U$ of $M$ is called a {\em pre-branch} if for some $t \in T$, $U \subseteq S_t$, and 
$\{[w]_t: w \in U\}$ is a branch of $B_t$ (at some point of $B(t)$).

Analogues of parts of Proposition 5.3 of \cite{Bhattacharjee} are then as follows. We recall that a {\em congruence}
of the action of a group on a set is an equivalence relation preserved by the group. There is an `improper'
congruence, in which all elements of the set are related, and if we say that a congruence is `maximal', we
understand that it is maximal subject to being proper.

\begin{lemma} \label{5.1} (i) For each $c \in C$, $G$ acts transitively on $\{[a]_t: a \in M, t \in T, F(t) = c\}$. 

(ii) $E_t$ is the unique maximal $G_t$-congruence on $S_t$; it is even the greatest proper $G_t$-congruence on $S_t$.

(iii) $G_t$ is transitive on the set of triples of distinct members of $\{[a]_t: a \in S_t\}$ for which $B_t$ holds.

\end{lemma}

\begin{proof} (i) This follows from Theorem \ref{4.6}(v). For singleton subtrees of $T$ (with singleton $B$-sets)
themselves form (degenerate, strongly embedded) finite trees of $B$-sets, and so two such having the same 
colour must lie in the same $G$-orbit.

(ii) As the $[a]_t$ form a partition of $S_t$ it is clear that $E_t$ is an equivalence relation on $S_t$. By its 
definition it follows that $E_t$ is preserved by the action of $G_t$ on $S_t$. To show that $E_t$ is a maximal 
$G_t$-congruence on $S_t$, we need to show that $G_t$ acts primitively on $B(t)$. Rather we verify the stronger
statement that it acts 2-transitively. Note that any two 2-element substructures of a $B$-relation are automatically 
isomorphic, since $B$ is ternary. For an arbitrary $B$-relation, it is not the case that this isomorphism need
extend to an automorphism. However our structure is sufficiently homogeneous to ensure that this extension 
exists. Formally this is again done by an appeal to Theorem \ref{4.6}(v). We give some more precise details. 

If $t$ is maximal, then the $B$-relation on $B(t)$ is isomorphic to the linear betweenness relation on $\mathbb Q$, 
on which $G_t$ acts 2-transitively. To deal with other cases, we first note that if $t$ is not maximal, and $c$
is some colour greater than the colour of $t$, then for any $[a]_t \in B(t)$ by considering a star at a node
coloured $F(t)$ below a node coloured $c$, and applying applying Theorem \ref{4.6}(iii) and (v), we see that
there is $s > t$ in $T$ coloured $c$ such that $f_t(s) = [a]_t$. Now let $[x]_t,[y]_t$ and $[x']_t,[y']_t$ be two 
pairs of distinct elements of $B(t)$. By what we have just remarked, there are $s_1, s_2, s_1', s_2'$ coloured
$c$ such that $f_t(s_1) = [x]_t$, $f_t(s_2) = [y]_t$, $f_t(s_1') = [x']_t$, and $f_t(s_2') = [y']_t$. Since 
$[x]_t, [y]_t, [x']_t, [y']_t$ must be ramification points of $B_t$, we can find isomorphic finite trees of 
$B$-sets having nodes $\{t, s_1, s_2\}$ and $\{t, s_1', s_2'\}$ (with leaves added at $[x]_t$, $[y]_t$, 
$[x']_t, [y']_t$ to fulfil the definition of `trees of $B$-sets'). By Theorem \ref{4.6}(v) once more, this
isomorphism extends to an automorphism of $A$, which lies in $G_t$ and takes $[x]_t$ to $[x']_t$ and $[y]_t$
to $[y']_t$. Since $G_t$ therefore acts 2-transitively on $B(t)$, and hence primitively, it follows that 
there is no $G_t$-congruence on $S_t$ properly containing $E_t$, so $E_t$ is maximal.

To see that $E_t$ is the unique maximal $G_t$-congruence, suppose that $\rho$ is a $G_t$-congruence on $S_t$ which
is not contained in $E_t$. The point is that $G_t$ acts transitively on each equivalence class $[x]_t$ of $E_t$. In 
fact, we appeal to a stronger statement, namely that the pointwise stabilizer of $M \setminus [x]_t$ acts transitively
on $[x]_t$, which is proved in Theorem \ref{5.3}. Now as $\rho \not \subseteq E_t$, there are $\rho$-equivalent $x, y$ 
such that $[x]_t \neq [y]_t$. Consider any $x' \in [x]_t$. By the above remark, there is a member of $G_t$ taking $x$ 
to $x'$ and fixing $y$. As $\rho$ is a $G_t$-congruence, $x' \sim_\rho y$. It follows that $[x]_t$ is contained in the 
$\rho$-class $[x]_\rho$ containing $x$. Now consider any $z \in [x]_\rho$. Since $G_t$ acts transitively on $B(t)$, 
there is $g \in G_t$ taking $[x]_t$ to $[z]_t$. This must fix $[x]_\rho$, and therefore $[z]_t \subseteq [x]_\rho$. 
It follows that $[x]_\rho$ is a union of $E_t$-classes. Thus $E_t \subset \rho$, and as $E_t$ is maximal, $\rho$ 
must be the improper congruence. Therefore $E_t$ is the unique maximal $G_t$-congruence on $S_t$.   

(iii) This also follows by considering suitable finite trees of $B$-sets as for the arguments in (i). \end{proof}

Now we are able to recover many of the constructs derived in \cite{Bhattacharjee}, but this time more directly. 
One of their notations was $S_{xyz}$ where $L(x; y, z)$, which is the subset of $M$ comprising those points $w$ 
such that for some finite strong substructure $B$ of $A$, $L(x; y, z)$ is witnessed at the root of $B$, and $w$ 
lies in the $B$-set at that root. In our context we can describe this as the union of all the sets of the form 
$[a]_t$, where $t$ is the node of $T$ at which $L(x; y, z)$ is witnessed. Since we don't really need to use triples 
to get to this point, we write $S_{xyz}$ instead as $S_t$. We observe that if $s \le t$ in $T$, and $[a]_t$ exists, 
then so does $[a]_s$, and it follows that $S_s \supseteq S_t$ (as remarked above).

It is important that Lemma \ref{3.3} holds also in the current context (for the limit structures). 

\begin{lemma} \label{5.2} For any $a, b, c \in M$ for which $L(a; b, c)$, there is a unique $t \in T^A$ such that 
$a, b, c \in S_t$, $[a]_t$, $[b]_t$, and $[c]_t$ are distinct, and $B_t([a]_t; [b]_t, [c]_t)$. Furthermore, for any 
distinct $a, b, c \in M$, the $L$-relation holds between $a, b,$ and $c$ in some order. \end{lemma}

\begin{proof} We may find $i$ such that $a, b, c \in M_i$. By Lemma \ref{3.3} there is $t \in T_i$ as claimed. This 
then serves equally well in the limit structure.     \end{proof}

Our whole object is to demonstrate that we have constructed a Jordan group which preserves a limit of $B$-relations, 
and which does not preserve the other structures in the list of Jordan groups. A Jordan group is, by definition, a 
permutation group, and the group $G$ in this case is taken to be the automorphism group of $A$, considered in its 
action on the corresponding $L$-set, denoted by $M$. Since this is to be a Jordan group, we should find subsets of 
$M$ which are Jordan sets for this action. We take sets of the form $[a]_t$. 

\begin{theorem} \label{5.3} For any $t \in T$ and $a \in S_t$, $[a]_t$ is a Jordan set for $G$.  \end{theorem}

\begin{proof} We recall that this means that the pointwise stabilizer of $M \setminus [a]_t$ acts transitively on 
$[a]_t$. To discuss transitivity of the group on this subset, it is necessary to understand how $G$ acts on members 
of $M$. The first point to note is that the nodes $[a]_t$ are all subsets of $M$, and that if $s < t$ in $T$, and 
$[a]_t$ exists, then so does $[a]_s$, and $[a]_s \subset [a]_t$. (This is immediate for finite trees of $B$-sets, and 
carries over to the limit.) In particular, the set of nodes $s$ of $T$ such that $[a]_s$ exists is downwards closed. 
Let us write $T_a$ for the set of all $s \in T$ such that $[a]_s$ exists. We now see that 
$\{a\} = \bigcap_{s \in T_a}[a]_s$. The inclusion $\subseteq$ is immediate. Conversely, let $b \neq a$ in $M$. 
There must therefore be some $i$ such that $a$ and $b$ both lie in $M_i$, and as $M_i$ is a substructure of $M$, 
$a$ and $b$ are unequal there, at node $s$ say (the root of $T_i$). By definition, $[a]_s = \{a\}$ and 
$[b]_s = \{b\}$. Thus $[a]_s \cap [b]_s = \emptyset$. For any $j \ge i$, this will still be true, though the two 
sets increase in size, and hence also in $M$. It follows that $b \not \in \bigcap_{s \in T_a}[a]_s$.

Now we can use this representation to say how $G$ acts on $M$, provided that we know how it acts on the tree of 
$B$-sets. This is precisely done by `taking limits', meaning that $\theta(a) = b$ if and only if for some 
$t_1 \in T_a, t_2 \in T_b$, $\theta$ takes $\{t \in T_a: t \le t_1\}$ to $\{t \in T_b: t \le t_2\}$, and for 
each $t \in T_a$ such that $t \le t_1$, $\theta([a]_t) = [b]_t$.

Reverting to the main argument, we write $J$ for the chain of points of $T$ less than $t$. For each $s \in J$, let 
$r_s$ be chosen in $M$ so that $f_s(t) = [r_s]_s$. In the tree of $B$-sets, $[r_s]_s$ is the member of $B(s)$ which 
is removed in passing up the linear order $J$, and whose branches are coalesced at higher stages (this is easier to 
visualize if $s$ has an immediate successor in $J$, but is no less true in general). Note that in particular,
$[r_s]_s \cap S_t = \emptyset$. We consider equivalence relations $E_s$, already defined, whose classes are the 
$[x]_s$ as $x$ varies, except that we now write this for the restriction of the original relation to $S_t$. We 
note that if $s_1 < s_2$ in $J$, then $E_{s_1} \upharpoonright S_{s_2} \subset E_{s_2}$, since each member of
$E_{s_2}$ is a union of pre-branches of $S_{s_1}$ at $[r_{s_1}]_{s_1}$ and hence of infinitely many members of
$E_{s_1}$.

Next we observe that the set $S$ of all $[b]_s$ for $b \in [a]_t$ where $s \le t$ forms an upper semilinear order
when partially ordered by inclusion, with top element $[a]_t$. For suppose that $[b]_s \subseteq [c]_{s_1}, [d]_{s_2}$
where $b, c, d \in [a]_t$ and $s, s_1, s_2 \in J$, and assume that $s_1 \le s_2$. Then 
$[b]_{s_2} \subseteq [c]_{s_2}, [d]_{s_2}$, so $[c]_{s_2} \cap [d]_{s_2} \neq \emptyset$. But all nodes of $B(s_2)$
are pairwise disjoint, and so $[c]_{s_2} = [d]_{s_2}$, which implies that $[c]_{s_1} \subseteq [d]_{s_2}$.

Now each branch of $B_s$ at $[r_s]_s$ is isomorphic to the semilinear order $K$ which was introduced in section 2, 
with chains isomorphic to $\mathbb Q$ of positive type and infinite ramification order. In \cite{Bhattacharjee} it 
was remarked that $G_{(M \setminus [a]_t)}$ acts on $S$ as a generalized wreath product of copies of ${\rm Aut}(K)$, 
and this was used to show that it acts transitively on $[a]_t$. We carry this out here more directly.

Choose a coinitial $\omega$-sequence $(s_n: n \in \omega)$ in $J$, meaning that $s_0 > s_1 > s_2 > \ldots$ and
every member of $J$ lies above some $s_n$. Let $b, c \in [a]_t$, and we find $\theta \in G_{(M \setminus [a]_t)}$
taking $b$ to $c$. The first task is to show that there is $\theta_n \in G$ fixing $[a]_t$ and $[r_{s_n}]_{s_n}$
and taking $[b]_{s_n}$ to $[c]_{s_n}$. This is done by considering finite trees of $B$-sets each having three nodes,
$s_n$, $t$, $t_n$ and $s_n$, $t$, $t_n'$ respectively, where $t$, $t_n$ and $t_n'$ have the same colour. For by
Theorem \ref{5.1}(i), there is a member of $G$ taking $[r_{s_n}]_{s_n}$ to $[b]_{s_n}$ and we take $t_n$ to be
the image of $t$ under such a map, and similarly for $t_n'$. For the tree of $B$-sets on nodes $s_n$, $t$, $t_n$, 
the $B$-set at $s_n$ has vertices $[r_{s_n}]_{s_n}$ and $[b]_{s_n}$, as well as 4 other vertices to ensure that
$[r_{s_n}]_{s_n}$ and $[b]_{s_n}$ are both ramification points, a singleton $[a]_t$ at $t$, and another singleton 
at $t_n$. Similarly for the tree of $B$-sets on nodes $s_n$, $t$, $t_n'$ with $[b]_{s_n}$ replaced by $[c]_{s_n}$, 
and $t_n$ by $t_n'$. There is an isomorphism between these two trees of $B$-sets which fixes $s_n$ and $t$, and by 
Theorem \ref{4.6}(v) this extends to $\theta_n \in G$ taking the first to the second, which provides the map desired.

Now let $\theta$ be defined first on $S$ by 
$$\theta(x) = \left \{ \begin{array}{lll}
\theta_0(x) & \mbox{ if} & x \not \subseteq [b]_{s_0}  \\
\theta_{n+1}(x) & \mbox{ if} & x \subseteq [b]_{s_n} \mbox{ and } x \not \subseteq [b]_{s_{n+1}}
\end{array} \right.$$
which is obtained by `glueing' together these maps, which is an order-automorphism of $S$ since 
$\theta_n[b]_{s_n} = [c]_{s_n} = \theta_{n+1} [b]_{s_n}$. One can see that this works by subdividing $S$ into countably
many sets $\{x: x \not \subseteq [b]_{s_0}\}$ and 
$\{x: x \subseteq [b]_{s_n} \mbox{ and } x \not \subseteq [b]_{s_{n+1}}\}$ in the domain and 
$\{x: x \not \subseteq [c]_{s_0}\}$ and 
$\{x: x \subseteq [c]_{s_n} \mbox{ and } x \not \subseteq [c]_{s_{n+1}}\}$ in the range.

This $\theta$ is now extended to act on the whole of $M$ (equivalently $T$). By the initial remarks about the action of 
$G$ on $M$, we see that $\theta(b) = c$. In a similar way, it is already defined on the whole of $[a]_t$, since any
member $x$ of $[a]_t \setminus \{b\}$ satisfies $x \not \in [b]_{s_n}$ for some least $n$ and then its image under 
$\theta$ is given by looking at $\theta_n$. We extend by fixing all points of $M \setminus [a]_t$. What is required to
conclude the argument is to find an action of $\theta$ on $T$ which induces the map already chosen. (In fact, giving 
the action of $\theta$ on $T$ would be sufficient to say how it acts on $M$, since it must preserve the $f$ and $g$ 
functions.) Let $u$ be an arbitrary node of $T \setminus J$. If $t \le u$ then we let $\theta(u) = u$. This preserves 
all the tree of $B$-sets structure above $t$, since $[a]_t$ is fixed, as are all other nodes contained in $B(t)$, by 
definition. Otherwise, $s = t \curlywedge u < t$, so lies in $J$. We know that $\theta([b]_s) = [c]_s$, and we have 
to show how to map the tree (or forest, if $s$ has a successor) above $[b]_s$ to the tree above $[c]_s$. One of the 
key points is that all these trees are disjoint, and so we can consistently glue together isomorphisms independently. 
Now $G$ acts transitively on the points of $B(s)$, and therefore the tree of $B$-sets above $[b]_s$ is isomorphic to 
the tree of $B$-sets above $[c]_s$, so there is {\em some} isomorphism we can use. We see that whichever 
one we have chosen will work. The key point is that it will fix all points of $M \setminus [a]_t$.

To justify this last statement, we recall our remarks at the start about how the members of $G$ act on $M$. That made it 
clear that this is entirely determined by how points are mapped `far enough down'. Let $d$ be an arbitrary member of 
$M \setminus[a]_t$. By what we showed, $\theta(d) = d$ if and only if there are $t_1, t_2 \in T_d$ for which $\theta$ 
takes $\{s \in T_d: s \le t_1\}$ to $\{s \in T_d: s \le t_2\}$, and for each $s \in T_d$ such that $s \le t_1$, 
$\theta([d]_s) = [d]_s$. Since $T$ is a tree, below some node, all members of $T_d$ lie in $J$, and since we know that
these are all fixed by $\theta$, we can take $t_1 = t_2 = t$ to be any member of $T_d \cap J$, which verifies the desired 
condition.    \end{proof}

\begin{corollary} \label{5.4} Each pre-branch is a Jordan set for $G$.  \end{corollary}
\begin{proof} We follow the same idea as in the lemma. Let $U$ be a pre-branch of the $B$-relation $B_t$ at a node $[r]_t$. 
This time we let $J$ be the chain of members of $T$ which are less than or equal to $t$, and for each $s \in J$ choose 
$r_s$ so that $[r_s]_s = f_s(t)$. So $[r]_t = [r_t]_t$. As before, the pointwise stabilizer of $M \setminus U$ fixes each 
$r_s$ (since clearly $r_s \not \in U$). Running the same argument as in Theorem \ref{5.3}, for any $a$ and $b$ in $U$ 
there is an element $\theta$ of $G_{(M \setminus U)}$ taking $a$ to $b$.   \end{proof}

\begin{corollary} \label{5.5}  Each set $S_t$ is a Jordan set for $G$ and $E_t$ is a maximal congruence on $S_t$.
Furthermore, the pointwise stabilizer $G_{(M \setminus S_t)}$ induces a $2$-transitive group preserving a 
$B$-relation on $S_t/E_t$.   \end{corollary}

\begin{proof} For this we use \cite{Adeleke3} Lemma 3.2, which says that the union of a `connected' family of Jordan 
sets is a Jordan set, where `connected' means in the graph-theoretical sense under the relation of overlapping non-trivially. 
Clearly $S_t$ is the union of all the pre-branches (at varying nodes of $B_t$), so it suffices to observe that the 
family of all pre-branches is a directed set. For let $U$ and $V$ be pre-branches at nodes $[a]_t$ and $[b]_t$, and we 
shall find a pre-branch $W \supseteq U \cup V$. First suppose that $[a]_t = [b]_t$. Since there are infinitely many 
pre-branches at $[a]_t$, we may find a node $[c]_t \neq [a]_t$ lying in a branch not equal to those defined by $U$ or 
$V$. Let $W$ be the (unique) pre-branch at $[c]_t$ containing $a$ (and $b$). This contains $U \cup V$ as desired. If 
$[a]_t \neq [b]_t$, choose a node $[c]_t$ between them under the relation $B_t$, and let $[d]_t$ lie in a branch of 
$B_t$ at $[c]_t$ which does not contain $[a]_t$ or $[b]_t$, and let $W$ be the pre-branch at $[d]_t$ containing $c$. 
Then again $W \supseteq U \cup V$.

The fact that $E_t$ is a maximal congruence as stated was shown in Lemma \ref{5.1}(ii). The final statement follows on 
considering $B_t$.                \end{proof}

\begin{lemma} \label{5.6} $(G, M)$ is $2$-primitive. \end{lemma}

\begin{proof} By definition, this means that the stabilizer $G_a$ of any $a \in M$ is primitive. We follow
the method given in \cite{Bhattacharjee}, Lemmas 5.8, 6.1, and Corollary 6.2. We first apply Lemma 5.8 to
$K$, the semilinear order with chains isomorphic to $({\mathbb Q}, <)$ and infinite ramification order at
each vertex. In this semilinear order, all cones are Jordan sets for its automorphism group. Since it also 
acts transitively, we deduce that its automorphism group is primitive.

Next (following \cite{Bhattacharjee} Lemma 6.1) we show that if $a$ and $b$ are distinct members of $M$, then every 
$c \in M \setminus \{a\}$ lies in a pre-branch $U$ of $[a]_t$ for some $t$ for which $b \in U$. If $b = c$ this is
clear. Otherwise, by Lemma \ref{5.2} there is $s \in T$ such that $[a]_s$, $[b]_s$, and $[c]_s$ all exist and 
are distinct. If $[b]_s$ and $[c]_s$ lie in the same branch of $B_s$ at $[a]_s$, then $c \in U$ for the corresponding 
pre-branch $U$, so we may take $t = s$. Otherwise there is some $t < s$ and $[d]_t \in B(t)$ such that 
$[a]_t$, $[b]_t$, $[c]_t$ lie in distinct branches of $B_t$ at $[d]_t$, in which case $[b]_t$ and $[c]_t$ lie 
in the same branch of $B_t$ at $[a]_t$, and we may take $U$ to be the corresponding pre-branch. To sum up, 
$M \setminus \{a\} = \bigcup\{U: U \mbox{ is a prebranch of some } B(t) \mbox{ containing } b\}$.

Now following \cite{Bhattacharjee} Corollary 6.2, let $X$ be a non-trivial block of $G_a$ where $a \in M$, and let 
$b$ and $c$ be distinct members of $X$. As in the previous paragraph there is $t \in T$ such that $[a]_t$, $[b]_t$, and 
$[c]_t$ exist and are distinct. In some finite tree of $B$-sets $A_i$ approximating our limit tree of $B$-sets $A$ these
tree vertices exist and are distinct, and there is some $s < t$ and $[d]_s \in B(s)$ such that $[a]_s$, $[b]_s$, 
and $[c]_s$ lie in distinct branches of $B_s$ at $[d]_s$. Hence $[b]_s$ and $[c]_s$ lie in the same branch of
$B_s$ at $[a]_s$. Therefore, $b$ and $c$ lie in the same pre-branch of $B_s$ at $[a]_s$. By Lemma \ref{5.3},
$[b]_s$ and $[c]_s$ are Jordan sets, and hence they are both contained in $X$. Since the branch at $[a]_s$ is a 
copy of $K$, its automorphism group is primitive, from which it follows that the whole of the pre-branch at $[a]_s$ 
containing $b$ is contained in $X$. Since 
$M \setminus \{a\} = \bigcup\{U: U \mbox{ is a prebranch of some } B(t) \mbox{ containing } b\}$, we deduce that
$M \setminus \{a\} = X$ as required.

\end{proof} 

\begin{lemma} \label{5.7} For each $a \in M$, there is a $G_a$-invariant $C$-relation on $M \setminus \{a\}$. \end{lemma}

\begin{proof} We follow the proof of Corollary 6.3 in \cite{Bhattacharjee}.

Let $\mathcal F$ be the set of all pre-branches omitting $a$, of the pre-sets containing $a$. By Corollary \ref{5.4} 
all members of $\mathcal F$ are Jordan sets, and by the proof given in \cite{Bhattacharjee}, which applies to our
situation, $\mathcal F$ has no pairs $\Gamma, \Delta$, for which $\Gamma \setminus \Delta$, $\Delta \setminus \Gamma$, 
and $\Gamma \cap \Delta$ are all non-empty (which there are called `typical pairs') as is seen by splitting into 
the cases that they are pre-branches in the same $B$-set, comparable $B$-sets, or incomparable $B$-sets. 

According to the analysis in \cite{Adeleke4}, Theorem 34.4, we are now able to find a $G_a$-invariant $C$-relation on 
$M \setminus \{a\}$ by letting $C(x; y, z)$ if there is a member of $\mathcal F$ which $y$ and $z$ both
lie in but $x$ does not.    \end{proof}

\begin{theorem} \label{5.8} The group $G$ preserves a limit of $B$-relations on $M$. \end{theorem}

\begin{proof} We follow the method of \cite{Bhattacharjee}. We have a long list of conditions to verify, as given at the end of section 2. 

By Theorem \ref{5.3} $(G, M)$ is an infinite Jordan group. For $J$ we take a $C$-chain of $T$, and we note that as $T$ has no
least element, neither does $J$. For $t \in J$, let $(B(t), B_t)$ be the $B$-set in $A$ at $t$, and let $\Gamma_t$ be the corresponding 
pre-$B$-set (which equals $S_t$ in the earlier notation). Thus if $s < t$ in $J$, $\Gamma_s \supset \Gamma_t$. We let $H_t$ be 
the pointwise stabilizer in $G$ of $M \setminus \Gamma_t$. If $s < t$ in $J$, then it follows that
$M \setminus \Gamma_s \subset M \setminus \Gamma_t$, so $H_t \supset H_s$.

We verify the numbered conditions in turn:

(i) This follows from Lemma \ref{5.1}(ii) and Corollary \ref{5.5}.

(ii) The action of $H_t$ on $\Gamma_t/E_t$ is 2-transitive as may for instance be deduced from Lemma \ref{5.1}(iii). It
is not 3-transitive since a triple $(a, b, c)$ such that $b$ lies between $a$ and $c$ cannot be taken to a triple in which no
element lies between the other two. It is also a Jordan group with branches at any node being Jordan sets, and by Corollary \ref{5.5}
it preserves the betweenness relation $B_t$ on $\Gamma_t/E_t$.

(iii) This holds since every point of $M$ lies in some pre-set $\Gamma_t$.

(iv) First $(\bigcup_{t \in J}H_t, M)$ is a Jordan group. To see that it is transitive, take any $a$, $b \in M$. By (iii)
there is $t$ such that $a, b \in \Gamma_t$. By Corollary \ref{5.5}, $\Gamma_t$ is a Jordan set for $G$, so $H_t$ acts 
transitively on it, and therefore has a member taking $a$ to $b$, and hence so does the union. To see that it is a Jordan
group, we may take the same set $\Gamma_t$ as a Jordan set, for some fixed $t \in J$. Then as $\Gamma_t$ is a Jordan
set for $G$, $H_t$ acts transitively on $\Gamma_t$, and hence so does the pointwise stabilizer of $M \setminus \Gamma_t$ in
the union. 

To verify 2-primitivity, we use the argument of Lemma \ref{5.6}.

(v) follows from the fact that if $s \le t$ then $E_s \upharpoonright \Gamma_t \subseteq E_t$ as remarked during the proof of 
Lemma \ref{5.3}.

(vi) holds since distinct points of $M$ are separated far enough down $T$.

We note that in our case condition (vii) comes out in a simpler form, since necessarily $i = j$, as follows from the fact that all 
members of $G$ preserve colours, namely 

$(\forall g \in G)(\exists i_0 \in J)(\forall i < i_0)(g(\Gamma_i) = \Gamma_i \wedge gH_ig^{-1} = H_i)$.

To verify this, given $g$, take any $t \in J$, and let $i_0 = t \curlywedge g(t)$. Then $g(i_0) = i_0$ since $g(i_0)$ lies in the set
of points of $T$ below $g(t)$, which is linearly ordered, and its unique member of the same colour as $i_0$ is $i_0$. Hence $g$ fixes
the chain in $T$ below $i_0$ pointwise, and the conclusion follows easily. 

(viii) follows from Lemma \ref{5.7}.    \end{proof}

\begin{lemma} \label{5.9} There is no $G$-invariant linear order, linear or general betweenness relation, circular order, separation 
relation, semilinear order,  Steiner system, or $C$- or $D$-relation on $M$. \end{lemma}

\begin{proof} We refer to \cite{Adeleke2} Theorem 1.0.2, which lists the options for an infinite primitive Jordan group which is not 
highly transitive. We note that our permutation group $G$ is 2-primitive but not 3-transitive. In fact it isn't even 3-homogeneous
(unlike the group constructed in \cite{Bhattacharjee}). To see this, consider any three distinct members $a$, $b$, $c$ of $M$. Then
there is a unique $t \in T$ such that $[a]_t$, $[b]_t$, and $[c]_t$ exist, and one lies between the other two in $B_t$. It is 
impossible to take this triple $\{a, b, c\}$ to another triple for which the corresponding member of $T$ has a different colour, since
all members of $G$ are colour preserving.

Looking at the six clauses in \cite{Adeleke2} Theorem 1.0.2, we see that the only one which applies to this situation is (iv), for
which the possibilities listed are a dense circular order, a $D$-relation, and limits of $B$- or $D$-relations. To rule out a 
circular ordering, note that it is not possible to fix a point $a$, and interchange two others, $b$ and $c$, while preserving a 
circular ordering. However, in our structure, it is possible to fix a point $a$, and a tree vertex $t$, while choosing $b$ and $c$ 
so that no triple containing all of $[a]_t$, $[b]_t$ and $[c]_t$ are related in the $B$-relation $B_t$. Then (as with the proofs in 
Lemma \ref{5.1}) it is clear (by Theorem \ref{4.6}) that there is an automorphism fixing $a$ and $t$ and switching $b$ and $c$. The 
argument to show that a $D$-relation cannot arise is more complicated, but is just the same as in \cite{Bhattacharjee} Lemma 6.6, 
so we do not go over the details.   \end{proof}

The other case in the list, namely a limit of $D$-relations, is presumably also not preserved by  $G$, though we have not verified 
this (and we note that this is also not done in \cite{Bhattacharjee} or \cite{Adeleke1}).

\section{Connection with Adeleke's construction, and recovery of the tree}

In this section we show that Adeleke's construction \cite{Adeleke1} is a special case of the one given in this paper, and we discuss 
recognizing the tree inside the permutation group $G = {\rm Aut}(M)$, which is required in order to know that we have $2^{\aleph_0}$
non-isomorphic examples.

The key step in Adeleke's construction is the passage from a betweenness relation $(\Omega, B)$ on an infinite set $\Omega$ 
to another betweenness relation $(\Omega^+, B^+)$, which is an extension of $(\Omega, B)$. In fact he phrases everything 
in terms of a group that respects the original configuration, but what he does amounts to dealing with sets carrying 
betweenness relations (which are then respected by the relevant group).

The definition of $\Omega^+$ in terms of $\Omega$ is as follows. It consists of all finite sequences of the form
$q_1 \omega_1 q_2 \omega_2 \ldots q_{k-1} \omega_{k-1} q_k$ for $k \ge 1$, $q_1 > q_2 > \ldots > q_k$ in $\mathbb Q$, and 
$\omega_1, \ldots, \omega_k$ members of $\Omega \setminus \{\alpha, \beta\}$ where $\alpha$ and $\beta$ are fixed and distinct 
members of $\Omega$. Since the group acting on $\Omega$ will be 2-transitive, it doesn't matter which $\alpha$ and $\beta$ are
taken. This family is then turned into an upper semilinearly ordered set by letting 
$p_1 \eta_1 p_2 \eta_2 \ldots p_{k-1} \eta_{k-1} p_k \le q_1 \omega_1 q_2 \omega_2 \ldots q_{l-1} \omega_{l-1} q_l$ if $l \le k$,
$p_i = q_i$ and $\eta_i = \omega_i$ for $1 \le i < l$, and $p_l \le q_l$. 

To gain an intuition for what this ordering looks like, we observe that 
$q_1 > q_1 \omega_1 q_2 > q_1 \omega_1 q_2 \omega_2 q_3 > \ldots > q_1 \omega_1 q_2 \omega_2 \ldots q_{k-1} \omega_{k-1} q_k$. Also,
$q_1 \omega_1 q_2 \omega_2 \ldots q_{k-1} \omega_{k-1} q_k \ge q_1 \omega_1 q_2 \omega_2 \ldots q_{k-1} \omega_{k-1} q_k'$ if
$q_k \ge q_k'$. From the definition we can verify that it is indeed an upper semilinear order (meaning that the set of points above any 
particular point is linearly ordered). Furthermore, in the terminology of \cite{Droste}, it is a tree (more accurately, an
`upside-down' tree) in which all maximal chains are isomorphic to $\mathbb Q$. We also verify that it is of `positive type', meaning 
that it ramifies precisely at the points of the structure. To see this, one notes that if $\eta \neq \eta'$ in 
$\Omega \setminus \{\alpha, \beta\}$, then $q_1 \omega_1 q_2 \omega_2 \ldots q_{k-1} \omega_{k-1} q_k$ is the least upper bound of
$q_1 \omega_1 q_2 \omega_2 \ldots q_{k-1} \omega_{k-1} q_k \eta q_{k+1}$ and 
$q_1 \omega_1 q_2 \omega_2 \ldots q_{k-1} \omega_{k-1} q_k \eta' q_{k+1}$. Thus it is actually a 2-homogeneous tree of `positive type',
in the terminology of \cite{Droste}, which means that we already have a very good visualization of what it looks like.

Next we examine how it ramifies, and explain why the two points $\alpha$ and $\beta$ were omitted. We already `know' about the 
ramification of $q_1 \omega_1 q_2 \omega_2 \ldots q_{k-1} \omega_{k-1} q_k$ to 
$q_1 \omega_1 q_2 \omega_2 \ldots q_{k-1} \omega_{k-1} q_k \eta q_{k+1}$ by using a member $\eta$ of $\Omega \setminus \{\alpha, \beta\}$.
All other points of $\Omega^+$ below $q_1 \omega_1 q_2 \omega_2 \ldots q_{k-1} \omega_{k-1} q_k$ are of the form
$q_1 \omega_1 q_2 \omega_2 \ldots q_{k-1} \omega_{k-1} q_k' \omega_k q_{k+1} \ldots q_l$ for some $q_k' < q_k$. The set of
such elements forms a cone at $q_1 \omega_1 q_2 \omega_2 \ldots q_{k-1} \omega_{k-1} q_k$, and in \cite{Adeleke1} this
is indexed by $\beta$. Finally, since we are interested not so much in the semilinear order, as the betweenness relation
thereby determined, we should also take account of the upwards direction, corresponding to $\alpha$. 

To sum up, $\Omega^+$ is the set of sequences just described, and $B^+$ is the betweenness relation on $\Omega^+$ derived
from the semilinear order. This betweenness relation is of positive type, it is countable and dense without leaves, and at
each node it ramifies, where the cones (which should now be called `branches' in the light of our usage earlier in the paper) 
are naturally indexed by $\Omega$.

It should now be clear that we are approaching the set-up in the earlier parts of the paper, and in fact we can derive a
tree of $B$-sets of height 2. At the root is placed $(\Omega^+, B^+)$. Each node $x$ of $\Omega^+$ gives rise to some $y$ on 
the next level up, and the points in the $B$-set at $y$ are precisely the branches of $B^+$ at $x$. In other words, a copy
of $\Omega$, which is endowed with $B$ as its betweenness relation. 

Next, Adeleke iterates the construction. So clearly for us this means that the tree of $B$-sets is extended downwards, in
$\omega$ steps. The resulting tree therefore has the form $T_C$ where $C$ is the chain $\omega^*$. In \cite{Bradley} this
is called the ${\mathbb N}^+$-tree. Note once again that there is a small difference in the treatment here (and in 
\cite{Bhattacharjee}) and in \cite{Bradley}. For us, the `top' level can be any linear betweenness relation, which matches 
up with Adeleke's  `top level' (he starts with the usual linear betweenness relation on $\mathbb Q$), whereas in \cite{Bradley} 
the tree of $B$-sets has two extra levels above this point, having a single point at the leaves, and two points one down.

Finally, we note that Adeleke spends quite a lot of time constructing groups at each stage. Thus he starts with $H$ on 
$\Omega$, which is assumed to preserve the betweenness relation there, and have various transitivity properties, and this
is extended to $H^+$ acting on $\Omega^+$. The groups that he uses are (presumably) not the whole automorphism groups of 
the structure up to that point, since at each stage, there is just a limited set of group elements added in forming the 
extension. If however, we always work with the group of {\em all} the automorphisms of the betweenness relation, then 
we don't need to study individual elements in such detail. The relevant transitivity properties have already been 
established in the previous section.

In summary, this shows that the construction we have presented extends and generalizes both the method of Bhattacharjee 
and Macpherson in \cite{Bhattacharjee}, and that of Adeleke in \cite{Adeleke1}.

\vspace{.05in}

Finally we remark that there is a wide variety of variants of the construction, obtained by taking different values for
the chain $C$. There are certainly $2^{\aleph_0}$ possibilities, which one may realize for instance by taking combinations
of $\mathbb Z$ and $\mathbb Q$. We now show that these really are distinct, by recovering (in a strong enough sense) the 
underlying tree from the permutation group. The first step is to show how to recover the $L$-relation from the permutation group.

\begin{lemma} \label{6.1} Let $G$ be the automorphism group of the tree of $B$-sets $A$ as given in Theorem \ref{4.6}, 
viewed as a permutation group of the domain $M = M^A$ of the corresponding $L$-set. Then $L$ is uniquely expressible
as the union of the family of $G$-orbits of ordered triples of distinct elements of $M$ which are symmetric in
their second and third co-ordinates.  \end{lemma}

\begin{proof} Here we are regarding $L$ as the set of ordered triples $(a, b, c)$ of distinct elements of $M$ such that
$L(a; b, c)$. By the properties of $L$-relations, $L(a; b, c) \leftrightarrow L(a; c, b)$. By Lemma \ref{5.2}, for any 
distinct $a$, $b$, $c$ in $M$ there is a unique $t \in T$ such that $[a]_t$, $[b]_t$, and $[c]_t$ are defined and distinct, 
and one lies between the other two in $B_t$, and we write $t = t(a, b, c)$. We note that the colour $F(t)$ is preserved
by $G$, and furthermore, by Theorem \ref{4.6}(v), any two ordered triples of the form $(a_1, b_1, c_1)$, $(a_2, b_2, c_2)$ 
for which $t = t(a_1, b_1, c_1) = t(a_2, b_2, c_2)$ and $B_t(a_1; b_1, c_1)$, $B_t(a_2; b_2, c_2)$ lie in the same
$G$-orbit. 

Let $X^t = \{(a, b, c): a, b, c \mbox{ are distinct members of } M \mbox{ and } t(a, b, c) = t\}$. Then the above remarks
show that $X^t$ is the disjoint union of three $G$-orbits, which are $X^t_1 = \{(a, b, c) \in X^t: L(a; b, c)\}$,
$X^t_2 = \{(a, b, c) \in X^t: L(b; a, c)\}$, and $X^t_3 = \{(a, b, c) \in X^t: L(c; a, b)\}$. Of these three orbits, only
the first is symmetric between the second and third co-ordinates, since for example $L(b; a, c)$ holding would mean that 
$b$ is between $a$ and $c$ in $B_t$, which is certainly incompatible with $c$ being between $a$ and $b$. This shows that 
$X^t_1$ is the unique one of these three orbits which is symmetrical between the second and third co-ordinates, and this is
a restriction of $L$. Taking the union over all possibles values of $F(t)$ gives the desired statement. 
 \end{proof}

\begin{theorem}  \label{6.2} For any countable chains $C_1$ and $C_2$ without least elements, if $A_1$ and $A_2$
are the trees of $B$-sets resulting from our construction applied to the chains $C_1$ and $C_2$ respectively as colour 
sets, and the permutation groups $G_1 = {\rm Aut}(A_1)$ and $G_2 = {\rm Aut}(A_2)$ are isomorphic in their
actions on $M_1$ and $M_2$, then $T_{C_1} \cong T_{C_2}$ and $C_1 \cong C_2$.  \end{theorem}

\begin{proof} The hypothesis assures us that there is a bijection $\theta$ from $M_1$ to $M_2$ which induces an isomorphism 
from $(G_1, M_1)$ to $(G_2, M_2)$. By Lemma \ref{6.1} $L_1$ is uniquely expressible as the union of $G_1$-orbits of ordered
triples of distinct elements of $M_1$ which are symmetric in the 2nd and 3rd co-ordinates, with a similar statement
for $L_2$, and it follows that $\theta$ is also an isomorphism of $L$-structures.

The rest of the proof consists in constructing an interpretation of the trees of $B$-sets $A_1$ and $A_2$ from the 
corresponding $L$-relations, using an adaptation of the method given in \cite{Bhattacharjee}, and it will follow that
$\theta$ induces the desired isomorphism from $A_1$ to $A_2$. The method given is quite long and complicated, so we shall 
just focus on the main points and refer the reader to \cite{Bhattacharjee} pages 75-79 for full details. We work now 
with just {\em one} $L$-relation on a set $M$ arising from a tree of $B$-sets $A$.

The interpretation is done using triples $(x, y, z)$ of elements of $M$ which satisfy $L$, and the first stage is how 
to represent the pre-set $S_{xyz} = S_{t(x, y, z)}$ of all elements of $M$ which lie in the pre-set where $L(x; y, z)$ 
is witnessed (here we are using the notation of Lemma \ref{6.1}). For this, in \cite{Bhattacharjee} a formula 
$P(x, y, z, u, v, w)$ in six variables is given (involving $L$, and a quaternary predicate $L'$, which is however 
definable from $L$) which characterizes $t(x, y, z) = t(u, v, w)$. It is then remarked that $S_{xyz}$ is precisely 
equal to the set of all $w \in M$ for which $P(x, y, z, w, y, z) \vee P(x, y, z, x, w, y) \vee P(x, y, z, x, w, z)$.  

This recovers the tree $T$, and it remains to recover its ordering. The observation here is that $s \le t$ if and only
if for any $a$ and $b$ in $S_t$, if $[a]_t$ and $[b]_t$ are distinct, then $a$ and $b$ lie in $S_s$ and $[a]_s$ and
$[b]_s$ are also distinct (note that this is a strengthening of the condition that $S_t \subseteq S_s$). To carry this
out we have to know how to express $[a]_t \neq [b]_t$ for members $a$ and $b$ of $S_t$. This is also done in 
\cite{Bhattacharjee} page 76, where under (D), a formula is given to define the equivalence relation $E_t$ (or
$E_{xyz}$ as it is written there), so what is desired is indeed expressible. To verify the stated equivalence, first 
note that we have shown in section 5 that if $s \le t$ then $S_t \subseteq S_s$, and if $[a]_t$ and $[b]_t$ are distinct, 
then so are $[a]_s$ and $[b]_s$, since $[a]_s = [b]_s \Rightarrow [a]_t = g_{st}[a]_s = g_{st}[b]_s = [b]_t$. Now 
suppose that $s \not \le t$. If $t < s$ let $[c]_t = f_t(s)$, and pick distinct $[a]_t$, $[b]_t$ lying in the same 
branch of $B_t$ at $[c]_t$. Then $[a]_s = [b]_s$ but $[a]_t \neq [b]_t$. If however $s$ and $t$ are incomparable, 
let $u = s \curlywedge t$. Then $u < s, t$. Let $[c]_u = f_u(s)$ and pick distinct $[a]_t$ and $[b]_t$ in $B(t)$,
both unequal to $[c]_t$ (if this exists). Then $[a]_u$, $[b]_u$ and $[c]_u$ lie in distinct branches of $B_u$ at
$[d]_u = f_u(t)$, and we deduce that $[a]_u$ and $[b]_u$ lie in the same branch of $B_u$ at $[c]_u$, giving $[a]_s = [b]_s$
and $[a]_t \neq [b]_t$.

The fact that $C_1 \cong C_2$ follows since these are just (isomorphic to) the sets of levels of the two trees, ordered as induced from the ordering on the trees recovered above. (One can obtain the sets of levels directly as the families of orbits of the two permutation groups; however this does not of itself enable us to recover their ordering.)
\end{proof}

\begin{corollary} \label{6.3} There are $2^{\aleph_0}$ pairwise non-isomorphic (as permutation groups) irregular primitive Jordan groups on 
countable sets which preserve limits of betweenness relations.    \end{corollary}

\begin{proof}  From what we have just shown it suffices to remark that there are $2^{\aleph_0}$ pairwise non-isomorphic 
countable linear orders $C$ without least members, which can be used as colour sets for the trees $T_C$. \end{proof}

%\newpage

\noindent{\bf Open questions}

\noindent 1. The results presented show that we can form a tree of $B$-sets giving rise to a Jordan group which preserves a
limit of betweenness relations in the case where the ambient tree is levelled of a very particular form, so this is not a true
generalization of \cite{Bhattacharjee}, where the tree is definitely not levelled. Is it possible to find a more general
class of (countable) trees for which the method works, while relaxing the `levelled' hypothesis?

\vspace{.1in}

\noindent 2. A similar question can be asked where the tree does not have positive type.

\vspace{.1in}

\noindent 3. We would really like to know if whenever a Jordan group (let's say, on a countable set) preserves a limit of 
betweenness relations, is there a tree of $B$-sets which one can somehow recover?

\vspace{.1in}

\noindent 4. In \cite{Almazaydeh} a similar generic construction is given for trees of $D$-sets. Can that also be 
modified for a larger class of ambient trees?

\vspace{.1in}

\noindent 5. Presumably our Jordan group does not preserve a limit of $D$-relations (nor the one in \cite{Bhattacharjee}),
but we haven't checked this. Similarly, the Jordan group of \cite{Almazaydeh} presumably does not preserve a limit of 
$B$-relations. 

\vspace{.1in}

\noindent 6. A significant step in proving Theorem \ref{4.6} involved showing that the union of the sequence of strongly embedded finite trees of $B$-sets $A_1 \leq A_2 \leq ...$ constructed in the proof of Theorem \ref{4.6} is an (infinite) tree of $B$-sets, in the sense defined after Lemma \ref{3.2}. Is it the case that both (i) any countable tree of $B$-sets is the union of a chain of finite trees of $B$-sets; and (ii) the union of any countable chain of finite trees of $B$-sets is a tree of $B$-sets? If so, this would mean that the class of countable trees of $B$-sets is \emph{semi-algebroidal} in the sense described in \cite{Droste2} and \cite{Droste3}.

\end{document}